\documentclass[12pt]{amsart}
\usepackage[utf8]{inputenc}
\usepackage[english]{babel}
\usepackage{graphicx}
\usepackage{enumitem}
\graphicspath{{images/}{../images/}}
\usepackage{blindtext}
\usepackage{listings}
\usepackage{caption}
\usepackage{subcaption}
\usepackage{amsmath}
\usepackage{amsthm}
\usepackage{hyperref}
\usepackage{amsfonts}
\usepackage{tikz-cd}
\usepackage{amssymb}
\usepackage[colorinlistoftodos]{todonotes}
\usepackage[margin=1in]{geometry}

\usepackage{subfiles}

\def\wt{\operatorname{wt}}

\newcommand{\rr}{\mathbb{R}}
\newcommand{\zz}{\mathbb{Z}}
\newcommand{\mca}{\mathcal{A}}

\newcommand{\apoly}{\textbf{P}}
\newcommand{\dpoly}{\textbf{P}^\bullet}
\newcommand{\owt}{\overline{\wt}}

\newtheorem{theorem}{Theorem}[section]

\newtheorem{proposition}[theorem]{Proposition}
\newtheorem{lemma}[theorem]{Lemma}

\newtheorem{prop}[theorem]{Proposition} 
\newtheorem{cor}[theorem]{Corollary} 

\newtheorem{conj}[theorem]{Conjecture}
\newtheorem{thm}[theorem]{Theorem}

\theoremstyle{definition}
\newtheorem{definition}[theorem]{Definition}
\newtheorem{defn}[theorem]{Definition}

\newtheorem{observation}[theorem]{Observation}
\newtheorem{example}[theorem]{Example}
\newtheorem{remark}[theorem]{Remark}
\newtheorem{rmk}[theorem]{Remark}

\title{Saturation of Newton polytopes of type A and D cluster variables}
\author{Amal Mattoo}
\email{amal.mattoo@columbia.edu}

\author{Melissa Sherman-Bennett}
\email{msherben@umich.edu}

\begin{document}
	\begin{abstract}
		We study Newton polytopes for cluster variables in cluster algebras $\mca(\Sigma)$ of types A and D. A famous property of cluster algebras is the Laurent phenomenon: each cluster variable can be written as a Laurent polynomial in the cluster variables of the initial seed $\Sigma$. The cluster variable Newton polytopes are the Newton polytopes of these Laurent polynomials. We show that if $\Sigma$ has principal coefficients or boundary frozen variables, then all cluster variable Newton polytopes are saturated. We also characterize when these Newton polytopes are \emph{empty}; that is, when they have no non-vertex lattice points. 
	\end{abstract}
	\maketitle 
	\tableofcontents
	\section{Introduction}
		Cluster algebras are a class of commutative rings with a rich combinatorial structure, introduced by Fomin and Zelevinsky \cite{CA1}. A cluster algebra of rank $n$ includes two sets of distinguished elements which together generate the algebra: a finite set of \emph{frozen variables} and a (usually infinite) set of \emph{cluster variables}. Cluster variables are grouped together in overlapping $n$-subsets, called \emph{clusters}. A cluster together with the frozen variables is a \emph{seed}\footnote{technically, a seed additionally includes a rectangular extended exchange matrix; this will not matter for our purposes, as the cluster determines the seed for these cluster algebras \cite{CA2}}. The frozen variables, their inverses, and the cluster variables generate $\mca$ as an algebra. A key feature of cluster algebras is the \emph{Laurent phenomenon}: each cluster variable can be written as a Laurent polynomial in the variables of any cluster \cite{CA1}. 
		
		Here, we study Newton polytopes of Laurent polynomial expressions for cluster variables (or, briefly, ``cluster variable Newton polytopes").
		
		\begin{defn}
			Given a Laurent polynomial 
			\[f(x_1, \dots, x_d)= \sum_{\textbf{a} \in \zz^d} c_\textbf{a} x_1^{a_1} \cdots x_d^{a_d}\]
			the \emph{support} of $f$ is $\{\textbf{a} \in \zz^d:c_\textbf{a} \neq 0 \}$. The Newton polytope of $f$, denoted $N(f)$, is the convex hull of its support. A Newton polytope $N(f)$ is \emph{saturated} if every lattice point of $N(f)$ is in the support of $f$, and $N(f)$ is \emph{empty} if every lattice point of $N(f)$ is a vertex\footnote{A polytope which intersects a lattice exactly in its vertex set is sometimes also called ``lattice-free."}.
		\end{defn}
		Note that emptiness is strictly stronger than saturation. Both conditions, roughly, ensure that not too much information is lost in passing from $f$ to $N(f)$. See \cite{SatNewtonPoly} for a survey of saturated Newton polytopes in algebraic combinatorics. Empty polytopes have been studied in e.g. \cite{EmptyBK,EmptyDO,EmptyK}; examples include matroid polytopes (more generally, any polytope whose vertices are $0-1$ vectors) and Delauney polytopes.
		
		We will focus particularly on cluster variable Newton polytopes for type $A$ and $D$ cluster algebras. Such cluster algebras are \emph{finite type}, meaning that they have only finitely many cluster variables \cite{CA2}; they are also examples of cluster algebras from surfaces \cite{FST}. We study three different choices of frozen variables: no frozen variables, \emph{principal coefficients} (see \cite{CA4}), and \emph{boundary frozens} (where frozen variables correspond to the boundary segments of the relevant surface).
		
		
		Newton polytopes of cluster variables, and more generally the support of cluster variables, have been used to construct bases for cluster algebras in the rank 2 case. Sherman and Zelevinsky gave an explicit description of all Newton polytopes in rank 2 cluster algebras of finite and affine type, and used the Newton polytopes to construct canonical $\mathbb{Z}$-bases \cite{rank2}. Their approach inspired the definition of greedy bases for arbitrary rank 2 cluster algebras in \cite{greedyRank2}. Greedy basis elements are uniquely characterized (up to scalar multiples) by a condition on their supports. 
		
		While the greedy basis has not been defined for higher rank, one might hope to find similar support characterizations for other bases for cluster algebras. This is one motivation for studying Newton polytopes of cluster variables. Many known bases of cluster algebras include the cluster variables, so understanding the support of cluster variables is a first step toward understanding supports of arbitrary basis elements. Newton polytopes of cluster variables also feature centrally in positive tropicalizations of finite type cluster algebras \cite{AHHL,JLS,SW}: the positive tropicalization is the normal fan of the Minkowski sum of all cluster variable Newton polytopes. 
		
		Other work on cluster variable Newton polytopes includes \cite{rank3}, where Newton polytopes for rank 3 cluster algebras with no frozen variables were computed. They are weakly convex quadrilaterals that may or may not be saturated. Cluster variable Newton polytopes for cluster algebras with principal coefficients were studied by Fei in \cite{F-poly}, under the guise of $F$-polynomial Newton polytopes. Fei used representation theoretic methods to show principal coefficient cluster variable Newton polytopes are saturated when the initial seed is acyclic\footnote{Cluster variable Newton polytopes for principal coefficient cluster algebras are related to those for cluster algebras with arbitrary frozen variables by a linear map given by the columns of the extended exchange matrix. However, this map is not usually unimodular, so may not preserve saturation.}. He further conjectured that saturation holds in general. 
		
		\begin{conj}[\protect{\cite[Conjecture 5.3]{F-poly}}] \label{conj:saturation}
			Let $\mca(\Sigma)$ be a cluster algebra with principal coefficients at the seed $\Sigma$. Then the Newton polytopes of cluster variables, written as Laurent polynomials in $\Sigma$, are saturated.
		\end{conj}
		
		Here, we prove Fei's conjecture in types $A$ and $D$ for an arbitrary choice of initial seed $\Sigma$. Note that, as mentioned above, saturation for acyclic $\Sigma$ was previously proved by Fei; we give independent, non-representation theoretic proofs for this case. We also prove saturation for other choices of frozen variables.
		
		\begin{theorem}\label{sat-results}
			The Newton polytope of an arbitrary cluster variable in $\mca(\Sigma)$ (written as a Laurent polynomial in $\Sigma$) is saturated when $\mca(\Sigma)$ is one of the following cluster algebras:
			\begin{enumerate}
				\item type $A$ with boundary frozen variables at $\Sigma$ (Theorem \ref{AB-saturated})
				\item type $A$ with principal coefficients at $\Sigma$ (Theorem \ref{AP-saturated})
				\item type $A$ with no frozen variables (Theorem \ref{NFASat})
				\item type $D$ with boundary frozen variables at $\Sigma$ (Theorem \ref{DB-always})
				\item type $D$ with principal coefficients at $\Sigma$ (Theorem \ref{DP-saturated}).
			\end{enumerate}
			
		\end{theorem}
		
		We also investigate the emptiness of cluster variable Newton polytopes. One inspiration for this is work of Kalman \cite{Kalman}, who studied cluster variable Newton polytopes for type $A$ cluster algebras with boundary frozen variables. Among other things, he showed that each vector in the support of the cluster variable is a vertex of the Newton polytope. He further conjectured that the Newton polytope has no lattice points in its relative interior.

		We prove something stronger than Kalman's conjecture: we show that the Newton polytope is in fact empty. More generally, we characterize when cluster variable Newton polytopes are empty for the cluster algebras listed in Theorem~\ref{sat-results}. This characterization is easiest to state in type A.
		
		\begin{theorem}\label{empt-results}
			Let $\mca(\Sigma)$ be a type A cluster algebra with boundary frozens or principal coefficients at $\Sigma$. Then all cluster variable Newton polytopes are empty (Theorem \ref{AB-saturated} and Theorem \ref{AP-saturated}).
		\end{theorem}
		
		We characterize emptiness in the type $A$ no frozen variable case in Corollary~\ref{NFAEmpty}; in the type $D$ boundary frozen variable case in Theorem~\ref{DB-Emptiness}; and in the type $D$ principal coefficients case in Theorem~\ref{DP-emptiness}. In these cases, some cluster variable Newton polytopes are empty; roughly speaking, these cluster variables are far away from the initial seed.
		
		Our proofs of both saturation and emptiness results extensively use the combinatorics of snake graphs and their matchings, which give Laurent polynomial formulas for cluster variables \cite{MSW}. 
		
		The paper is organized as follows. Section 2 contains background information on cluster algebras of types $A$ and $D$. We also recall the results of \cite{MSW} on formulas for Laurent polynomial expressions for cluster variables using matchings of ``snake graphs'' (following \cite{MSW}). Section 3 introduces a number of polytopes of interest, and we prove some relevant lemmas. In Section 4, we prove saturation and emptiness results for type $A$ cluster algebras with boundary frozen variables, principal coefficients, and no frozen variables. In Section 5, we prove saturation and emptiness results for type $D$ cluster algebras with boundary frozen variables and principal coefficients. Finally, in Section 6 we present examples of un-saturated cluster variable Newton polytopes for cluster algebras from other surfaces, showing that many of our results are unlikely to extend to general surfaces. Nonetheless, we give a few conjectures about saturation for cluster algebras from surfaces. 
		
		\vspace{12pt}
		
		\noindent \textbf{Acknowledgements.} This work was carried out as part of the Herchel Smith-Harvard Undergraduate Science Research Program at Harvard University in Summer 2020. Both authors would like to thank Lauren Williams for helping to steer this research. M.S.B was supported by NSF Graduate Research Fellowship No. DGE-1752814.
	\section{Background}
		
		
		Cluster algebras of finite type (that is, with finitely many cluster variables) were classified in \cite{CA2}; their classification matches that of finite type Dynkin diagrams. We will focus exclusively on cluster algebras of types $A$ and $D$. 
		In general, a cluster algebra $\mca$ is defined recursively from an initial seed $\Sigma$ using a local move called \emph{mutation}. We will take a shortcut: for our cluster algebras of interest, rather than define mutation, we will provide explicit Laurent polynomial formulas for the cluster variables of $\mca$ in terms of $\Sigma$. Choosing a different initial seed will result in different Laurent polynomial formulas.
		
		As we will recall below, for type $A$ and $D$ cluster algebras, the combinatorics of clusters and cluster variables are encapsulated by triangulations of certain marked surfaces. We consider three different choices of frozen variables: no frozen variables; \emph{boundary frozen variables}, which correspond to the boundary segments of the surface; and \emph{principal coefficients} at the initial seed $\Sigma$ (see \cite{CA4}). We first review the combinatorial objects indexing seeds and cluster variables, and then, for each choice of frozen variables, give Laurent polynomial formulas for cluster variables.

		
		\subsection{Tagged arcs and triangulations of polygons and punctured polygons}
		
		For type $A$ and $D$ cluster algebras, cluster variables and extended clusters are in bijection with the tagged arcs and triangulations, respectively, of two particular marked surfaces. 
		
		\begin{definition}
			Let $S$ be a connected oriented surface with boundary (possibly empty), and let $M\subset S$ be a finite set of points with at least one on each boundary component. The elements of $M$ are \emph{marked points}, elements of $M$ in the interior of $S$ are \emph{punctures}, and $(S,M)$ is a \emph{marked surface}. 
		\end{definition}
		
		For a type $A$ cluster algebra, the relevant marked surface $\apoly$ is a polygon whose marked points are the vertices. For a type $D$ cluster algebra, the relevant marked surface $\dpoly$ is a polygon whose marked points are the vertices and a single puncture.
		
		\begin{definition}
			An \emph{(ordinary) arc} in a marked surface $(S,M)$ is a curve $\gamma\subset S$, considered up to isotopy, with endpoints in $M$ such that: other than its endpoints, $\gamma$ is disjoint from $M$ and $\partial S$; $\gamma$ does not intersect itself except possibly at endpoints; and $\gamma$ does not, by itself or together with a segment of the boundary, enclose an unpunctured monogon or unpunctured bigon.
		\end{definition}
		
		An arc of $\dpoly$ is a \emph{radius} if one endpoint is the puncture, and is a \emph{loop} if its endpoints coincide. For $\apoly$, arcs are just diagonals between two non-adjacent vertices.  
		
		\begin{defn}
			Two ordinary arcs in $(S,M)$ are \emph{compatible} if they (are isotopy equivalent to curves that) do not intersect except potentially at their endpoints. An \emph{ideal triangulation} of $(S,M)$ is a maximal collection of pairwise compatible arcs. The arcs of an ideal triangulation cut $(S,M)$ into \emph{ideal triangles}.
		\end{defn}
		
		An ideal triangulation of $\apoly$ is just a usual triangulation, and each ideal triangle has three distinct sides. An ideal triangulation of $\dpoly$ may include a \emph{self-folded triangle}, which is a loop enclosing a radius (see Figure~\ref{fig:triangulation-example} for an example). 
		
		\begin{figure}
			\centering
			\includegraphics[width=0.9\linewidth]{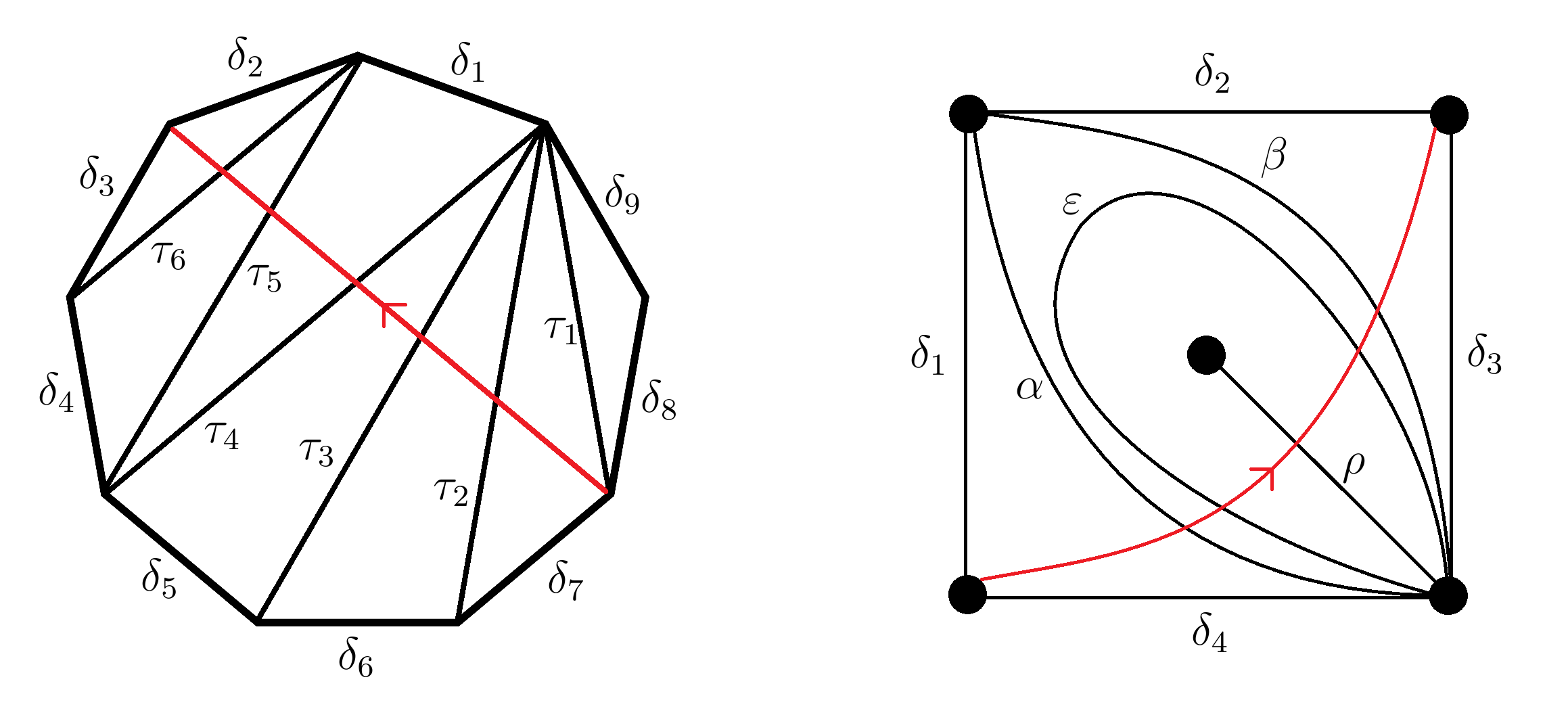}
			\caption{An ideal triangulation of a polygon (left) and a punctured polygon (right) and an arc, shown in red. Note that the triangulation of the punctured polygon includes a loop enclosing a radius.}
			\label{fig:triangulation-example}
		\end{figure}
		
		To obtain indexing sets for cluster variables and clusters, we need a generalization of arcs and ideal triangulations involving tagging.
		
		\begin{defn}
			Let $(S,M)$ be $\apoly$ or $\dpoly$. A \emph{tagged arc} $\gamma$ in $(S,M)$ is either an ordinary arc with distinct endpoints (which we say is tagged \emph{plain}) or a radius which is tagged \emph{notched}. We write $\rho^{\bowtie}$ for the notched version of the radius $\rho$. We denote by $\gamma^\circ$ the ordinary arc obtained by ignoring the tagging of $\gamma$.
		\end{defn}
		
		\begin{defn}
			Tagged arcs $\alpha$ and $\beta$ are \emph{compatible} if $\alpha^\circ$ and $\beta^\circ$ are compatible as ordinary arcs ($\alpha^\circ$ and $\beta^\circ$ may coincide) and if $\alpha^\circ, \beta^\circ$ are distinct radii, then $\alpha$ and $\beta$ have the same tagging. A \emph{tagged triangulation} is a maximal collection of pairwise compatible tagged arcs. 
		\end{defn}
		
		Note that for $\apoly$, tagged arcs and ordinary arcs coincide, so tagged triangulations are the same as usual triangulations. In a tagged triangulation of $\dpoly$, either all radii have the same tagging or there are exactly two tagged radii, which are $\rho$ and $\rho^{\bowtie}$.
		
		The following result shows how tagged arcs and triangulations relate to cluster variables and seeds.
		
		\begin{thm}\cite{FST} \label{thm:triangulationsToClusters}
			Let $\mca$ be a cluster algebra of type $A$ (resp. $D$) with frozen variables $F$. Then the cluster variables of $\mca$ are in bijection with tagged arcs of $\apoly$ (resp. $\dpoly$). Writing $x_\tau$ for the cluster variable corresponding to arc $\tau$, the map \[T \mapsto \Sigma_T:=\{x_\tau\}_{\tau \in T} \cup F\] is a bijection between tagged triangulations of $\apoly$ (resp. $\dpoly$) and seeds of $\mca$. 
		\end{thm}
		
		Given $T$ a tagged triangulation of $\apoly$ (resp. $\dpoly$), we will use the following notation:
		\begin{itemize}
			\item$\mca^{\text{bd}}(T)$ is the cluster algebra of type $A$ (resp. $D$) with boundary frozen variables at the seed $\Sigma_T$. The frozen variables are in bijection with the boundary arcs of the surface\footnote{More precisely, the exchange matrix of $\Sigma_T$ has columns indexed by $T$ and rows indexed by $T$ and the boundary arcs. The entries reflect adjacency of arcs, and are computed according to \cite[Definition 3.10]{FST}}.
			\item$\mca^{\text{pc}}(T)$ is the cluster algebra of type $A$ (resp. $D$) with principal coefficients at the seed $\Sigma_T$ \cite[Definition 3.1]{CA4}. There is one frozen variable for each arc of $T$.
			\item $\mca^{\text{nf}}(T)$, which has no frozen variables.
		\end{itemize}
		
		Next, we introduce notation for the Laurent polynomial expressions for cluster variables in these cluster algebras.
		\begin{defn}
			Let $(S, M)$ be $\apoly$ (resp. $\dpoly$), with boundary segments $Z=\{\zeta_1, \dots, \zeta_r\}$, let $T=\{\tau_1, \dots, \tau_n\}$ be a tagged triangulation, and let $\gamma$ be a tagged arc. 
			\begin{itemize}
				\item For $\mca^{\text{bd}}(T)$, denote by $L_{T, \gamma}^{\text{bd}}:=L_{T, \gamma}^{\text{bd}}(x_{\tau_1},...,x_{\tau_n},x_{\zeta_1},...,x_{\zeta_r})$ the Laurent polynomial expansion for cluster variable $x_\gamma$ in $\Sigma_T$.
				\item For $\mca^{\text{pc}}(T)$, denote by $L_{T, \gamma}^{\text{pc}}:=L_{T, \gamma}^{\text{pc}}(w_{\tau_1},...,w_{\tau_n},y_{\tau_1},...,y_{\tau_{n}})$ the Laurent polynomial expansion for cluster variable $w_\gamma$ in $\Sigma_T$. Note that $w_{\tau_i}$ is a cluster variable and $y_{\tau_i}$ is a frozen variable.
				\item For $\mca^{\text{nf}}(T)$, denote by $L_{T, \gamma}^{\text{nf}}:=L_{T, \gamma}^{\text{nf}}(z_{\tau_1},...,z_{\tau_n})$ the Laurent polynomial expansion for cluster variable $z_\gamma$ in $\Sigma_T$.
			\end{itemize}
			We use $L_{T,\gamma}$ to denote the Laurent polynomial without specifying the choice of frozen variables.
		\end{defn}
		
		For every tagged triangulation $T$ and tagged arc $\gamma$, \cite{MSW} gives a formula for $L_{T, \gamma}$ using matchings of snake graphs. We now review their constructions, beginning with snake graphs.
		
		\subsection{Snake graphs}
		
		Fix a surface $(S,M)$ with boundary segments $Z=\{\zeta_1, \dots, \zeta_r\}$, and an ideal triangulation $T= \{\tau_1, \dots, \tau_n\}$. For each choice of arc $\gamma$ in $(S, M)$, we will build a graph consisting of glued-together squares which ``snakes" weakly north-west in the plane. We first define the basic building blocks of this graph.
		
		For an ideal triangle $A$ of $S\setminus T$ which is not self-folded, define $\Delta$ to be the triangle graph with edges labeled by the three distinct arcs of $T$ bounding $A$. If $A$ is self-folded, then define $\Delta$ to be the triangle graph with one edge labeled by the loop $\lambda$ bounding $A$ and the other two edges labeled by the radius enclosed by $\lambda$.
		
		For $\tau \in T$, $\tau$ is in the boundary of two ideal triangles $A_1$ and $A_2$. The \emph{tile} $G_{\tau}$ is obtained by gluing together $\Delta_1$ and $\Delta_2$ along the edge labeled $\tau$ so that the orientations of $\Delta_1$ and $\Delta_2$ either both agree or both disagree with those of $A_1$ and $A_2$ (see Figure~\ref{fig:gluing-tiles}). This gives two possible planar embeddings of $G_\tau$. If $\tau$ is a radius enclosed in a loop $\lambda$, then $A_1=A_2$ and we glue $\Delta_1$ and $\Delta_2$ so that the edges labeled $\lambda$ are adjacent. 
		
		\begin{figure}
			\centering
			\includegraphics[width=0.4\linewidth]{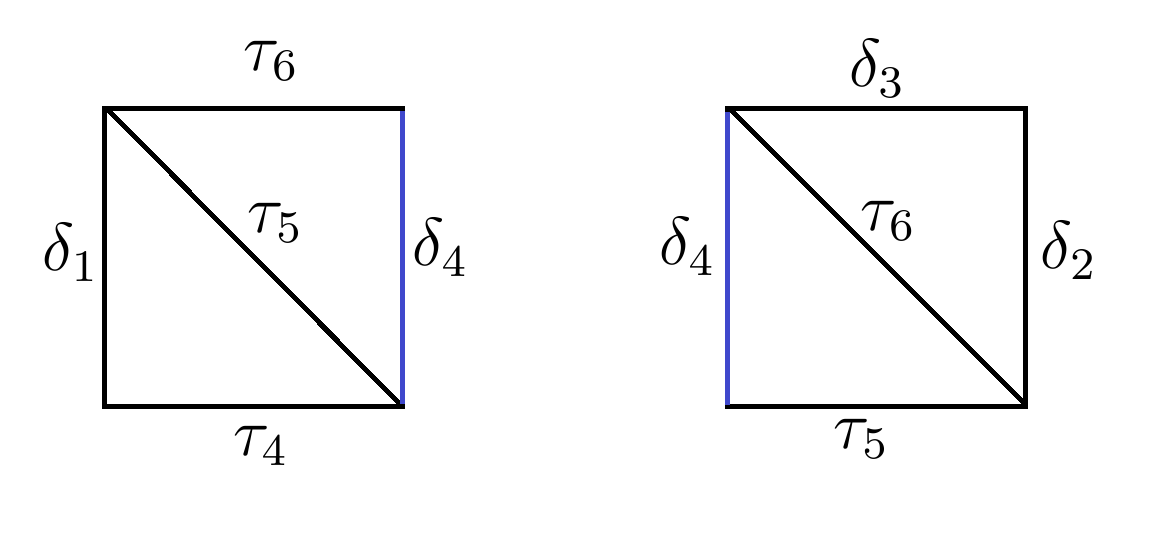}
			\caption{The tiles $G_{\tau_5}$ and $G_{\tau_6}$ for the triangulation $T$ on the left of Figure \ref{fig:triangulation-example}. They will be glued together along the blue arcs, both of which have label $\delta_4$, yielding the last two boxes of the snake graph in Figure \ref{fig:snake-graph-example} on the left.}
			\label{fig:gluing-tiles}
		\end{figure}
		
		Now, let $\gamma \notin T$ be an ordinary arc in $(S,M)$ and choose an orientation of $\gamma$. Let $\tau_{i_1},...,\tau_{i_k}$ (not necessarily distinct) be the sequence of arcs of $T$ that $\gamma$ intersects, and set $G_j:=G_{\tau_{i_j}}$. We construct the \emph{snake graph} $G_{T,\gamma}$ as follows. Choose the planar embedding of $G_1$ so that the orientation of its triangles agrees with those in $T$. For $j=2, \dots, k$, glue $G_j$ to $G_{j-1}$ along the edge with the shared label so that odd tiles have triangles oriented the same as in $T$, and even tiles have triangles oriented oppositely (see Figure~\ref{fig:gluing-tiles}). This gluing is unambiguous except when $\tau_{i_j}$ is a radius enclosed in a loop $\lambda$; in this case, glue $G_{j-1}, G_j, G_{j+1}$ as illustrated in Figure~\ref{fig:selfFoldedTile}. Then remove the diagonal from each tile to yield the snake graph $G_{T,\gamma}$ (see Figure~\ref{fig:snake-graph-example}). 
		
		\begin{figure}
			\centering
			\includegraphics[height=0.35\textheight]{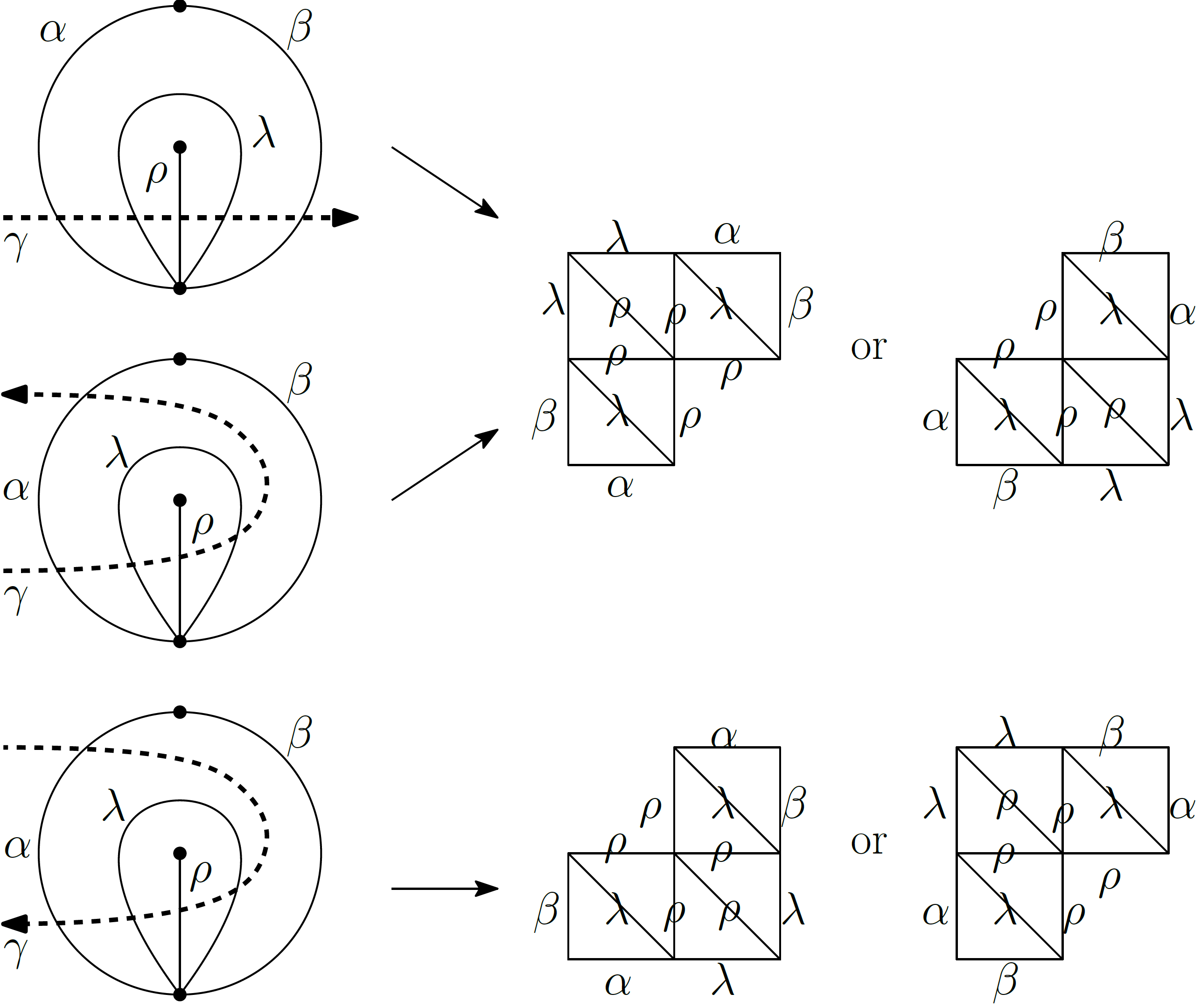}
			\caption{\label{fig:selfFoldedTile} How to glue tiles $G_\lambda, G_\rho, G_\lambda$ when $\gamma$ crosses a self-folded triangle. Note that if $T$ contains a loop $\lambda$, then $\lambda$ is enclosed in a bigon, as pictured here. The snake graphs on the left show the case when $G_\lambda$ is oriented as in $T$; on the right, the case when $G_\lambda$ has the opposite orientation.}
		\end{figure}

		\begin{figure}
			\centering
			\includegraphics[width=0.9\linewidth]{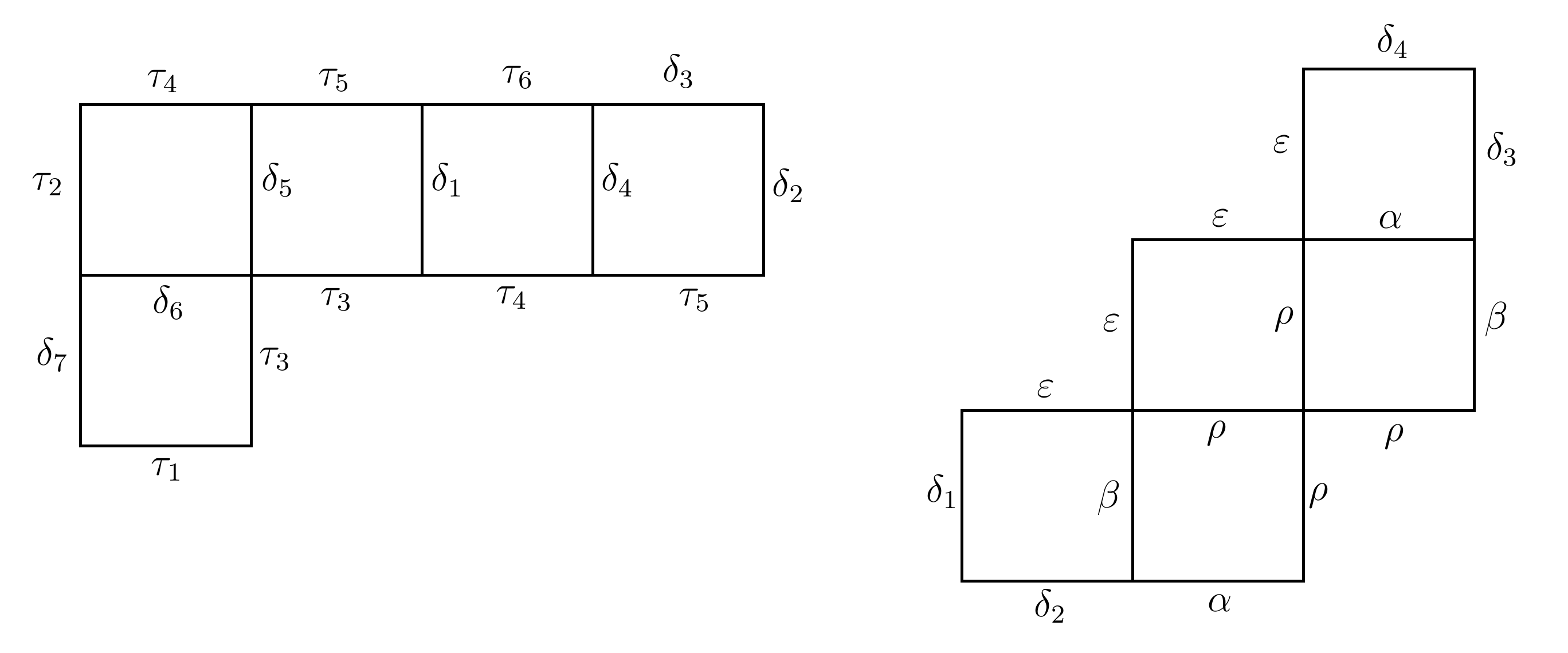}
			\caption{Snake graphs corresponding to the red arcs and triangulations in Figure \ref{fig:triangulation-example}, with the snake graph corresponding to the polygon on the left and the punctured polygon on the right.}
			\label{fig:snake-graph-example}
		\end{figure}

		We call $G_\tau$ with its diagonal removed a \emph{square} or a \emph{tile} of $G_{T, \gamma}$. For an edge $e$ of $G_{T, \gamma}$, we denote its label by $\ell(e)$, and for a tile $G_{\tau}$ let $\ell(G_{\tau}):=\tau$.

		\begin{rmk}
			Abusing notation, we will also use the term \emph{snake graph} to refer to an edge-labeled graph $G$ which is isomorphic to some $G_{T, \gamma}$ as an unlabeled graph. Again, for $e$ an edge of $G$, we use $\ell(e)$ to denote its label.
		\end{rmk}
		
		We are interested in weight vectors of perfect matchings of $G_{T, \gamma}$. There will be three different kinds of weight vectors, one for each choice of frozen variables. We first need some preliminary definitions.
		
		\begin{defn}
			A \emph{perfect matching} $M$ of a graph $G$ is a subset of edges such that each vertex of $G$ is in exactly one edge of $M$. The \emph{bottom matching} $M_0$ of a snake graph $G$ is the matching which only involves boundary edges of $G$ and contains the south edge of the first tile. 
		\end{defn}
		
		See Figure~\ref{fig:matching-ex} for an example.
		
		\begin{lemma}
			\cite[Lemma 4.7]{MSW} For any matching $M$ of $G_{T,\gamma}$, the symmetric difference $M\ominus M_0$ encloses a set of tiles of $G_{T,\gamma}$. 
		\end{lemma}
		
		For a matching $M$ of $G_{T,\gamma}$, let $C(M)$ denote the set of tiles enclosed by $M\ominus M_0$ (see Figure~\ref{fig:matching-ex}).
		
		\begin{defn}\label{def:matching}
			Let $G_{T, \gamma}$ be a snake graph and let $M$ be a perfect matching. We define the following \emph{weights} of $M$:
			\begin{align*}
				\wt^{\text{bd}}(M)&:=\Pi_{e \in M} x_{\ell(e)} \\
				\wt^{\text{nf}}(M)&:=\Pi_{e \in M} z_{\ell(e)}|_{z_{\zeta_i}=1}\\
				\wt^{\text{pc}}(M)&:=\Pi_{e \in M} w_{\ell(e)}|_{w_{\zeta_i}=1}\Pi_{s\in C(M)}y_{\ell(s)}
			\end{align*} 
			with $\wt(M)$ denoting the weight when the choice of frozen variables has not been specified.
			
			We denote the exponent vectors of these weights by 
			\begin{align*}
				w_{\text{bd}}^M &\in\rr^{\tau_1,...,\tau_n,\zeta_1,...,\zeta_r}\\
				w_{\text{nf}}^M &\in\rr^{\tau_1,...,\tau_n}\\
				w_{\text{pc}}^M &\in\rr^{\tau_1,...,\tau_n}\times\rr^{\tau_1,...,\tau_n}
			\end{align*}
			respectively, and call them
			the \emph{weight vectors} of $M$.
		\end{defn}
		We define weights and weight vectors of matchings analogously for snake graphs $H$ which are subgraphs of $G_{T, \gamma}$. See Figure~\ref{fig:matching-ex} for examples of weights and weight vectors.
		
		\begin{figure}
			\centering
			\begin{minipage}{0.3\linewidth}
				\includegraphics[width=\linewidth]{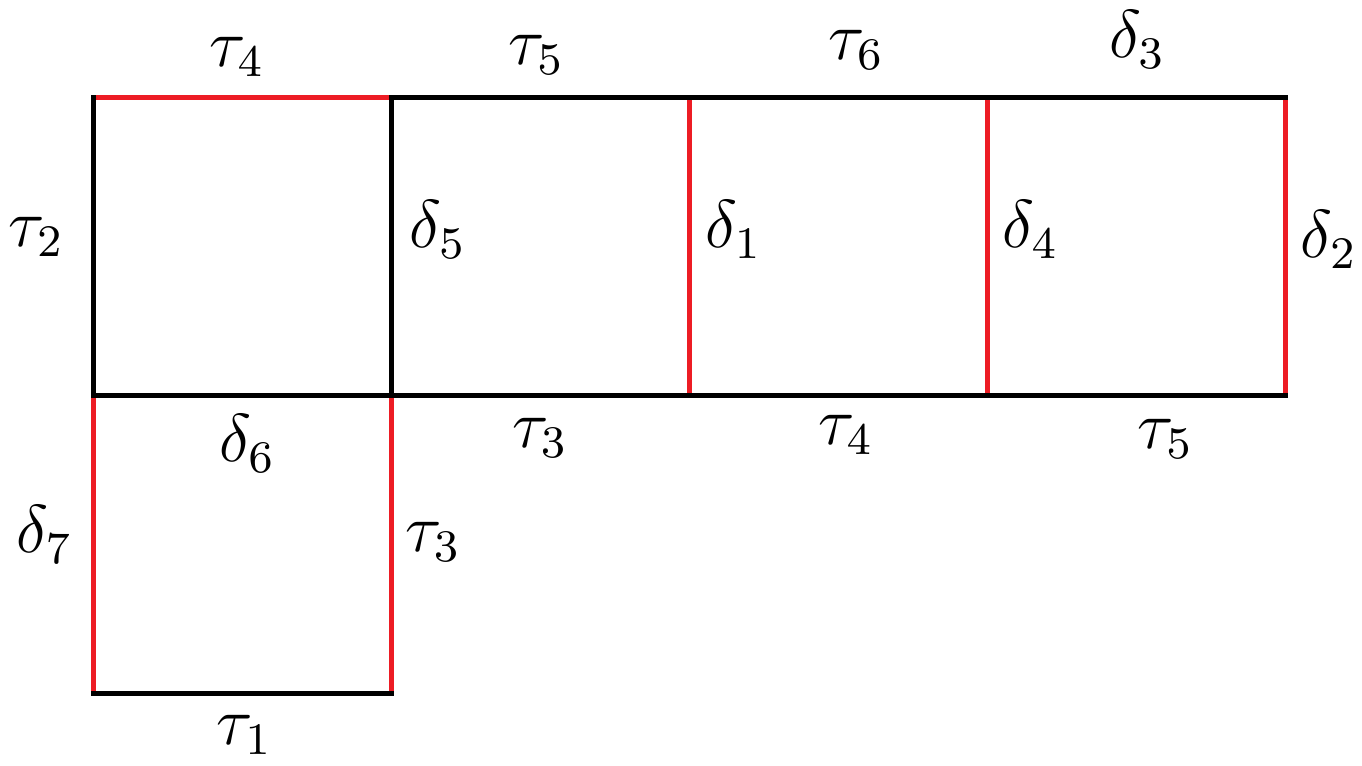}
				$$\mbox{\small $\text{wt}^{\text{bd}}(M)=
					x_{\tau_{3}}x_{\tau_{4}}x_{\delta_{1}}x_{\delta_{2}}x_{\delta_{4}}x_{\delta_{7}}$}$$
				$$\mbox{\small $w^M_{\text{bd}}=$ \tiny $(0,0,1,1,0,0;1,1,0,1,0,0,1,0,0)$} $$ 
			\end{minipage}
			\begin{minipage}{0.3\linewidth}
				\includegraphics[width=\linewidth]{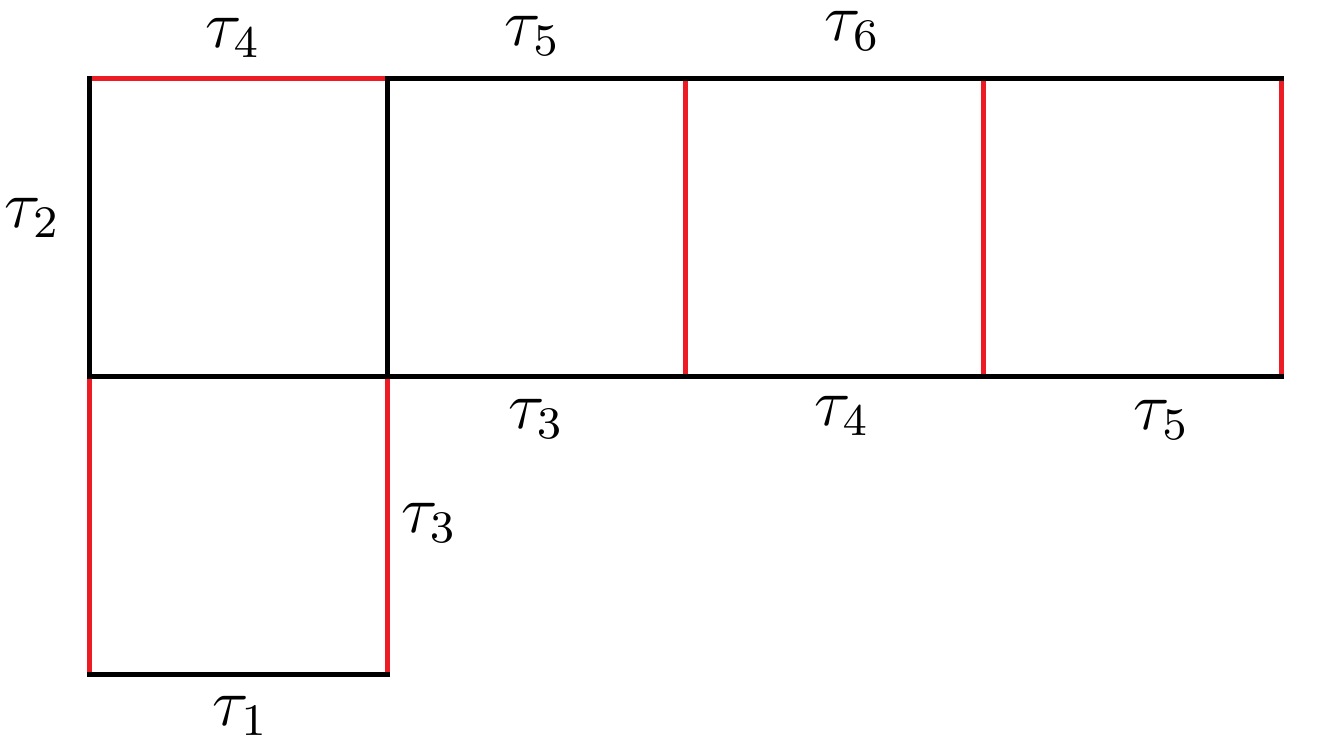}
				$$\mbox{\small$\text{wt}^{\text{nf}}(M)=
					z_{\tau_{3}}z_{\tau_{4}}$}$$
				$$\mbox{\small $w^M_{\text{nf}}=$\tiny$(0,0,1,1,0,0)$}$$
			\end{minipage}
			\begin{minipage}{0.3\linewidth}
				\includegraphics[width=\linewidth]{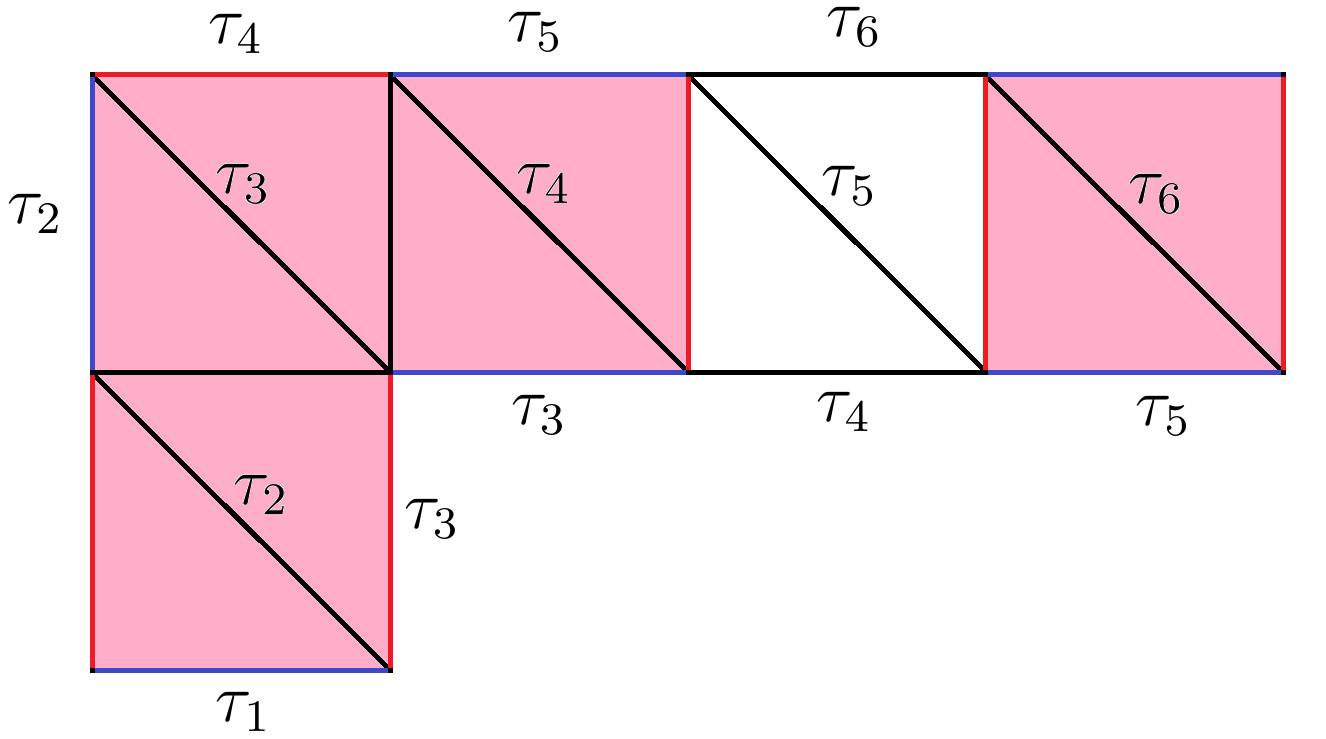}
				$$\mbox{\small$\text{wt}^{\text{pc}}(M)=
					w_{\tau_{3}}w_{\tau_{4}}y_{\tau_{2}}y_{\tau_{3}}y_{\tau_{4}}y_{\tau_{6}}$}$$
				$$\mbox{\small$w^M_{\text{pc}}=$\tiny$(0,0,1,1,0,0;0,1,1,1,0,1)$}$$
			\end{minipage}
			\caption{A matching $M$ (in red) of the snake graph from Figure \ref{fig:snake-graph-example} on the left, and the weights and weight vectors for the three choices of frozen variables. On the far right, the bottom matching is in blue, and the squares in $C(M)$ are highlighted in pink.}
			\label{fig:matching-ex}
		\end{figure}
		
		We will need the following notion in our proofs.
		
		\begin{defn} \label{def:corner}
			Let $G$ be a snake graph. A \emph{corner square} of $G$ is a square (not the first or last) whose neighboring squares share a vertex with each other. All other squares are \emph{non-corner squares}.
		\end{defn}
		
		\subsection{Expansion formulas}
		
		In this section, we give the formulas of \cite{MSW} for $L_{T, \gamma}$ using matchings of snake graphs. First, we note that it suffices to provide $L_{T, \gamma}$ for tagged triangulations $T$ with no notched radii.
		
		\begin{proposition}\label{prop:flip}
			\cite[Proposition 3.16]{MSW}  Let $T=(\tau_1,...,\tau_n)$ be a tagged triangulation of $\dpoly$ or $\apoly$ and $\gamma$ a tagged arc. Let $\gamma^p$ be the arc obtained from $\gamma$ by changing the tagging if $\gamma^\circ$ is a radius; otherwise, set $\gamma^p=\gamma$. Let $T^p:=(\tau_1^p,...,\tau_n^p)$. Then 
			\begin{equation}\label{eq:flip}
				L^{\text{bd}}_{T, \gamma}= L^{\text{bd}}_{T^p,\gamma^p}|_{x_{\tau^p}\mapsto x_\tau}, \quad
				L^{\text{nf}}_{T, \gamma}= L^{\text{nf}}_{T^p,\gamma^p}|_{z_{\tau^p}\mapsto z_\tau}, \quad
				L^{\text{pc}}_{T, \gamma}= L^{\text{pc}}_{T^p,\gamma^p}|_{w_{\tau^p}\mapsto w_\tau,y_{\tau^p}\mapsto y_\tau}
			\end{equation}
			
		\end{proposition}
		
		Because of Proposition~\ref{prop:flip}, we will restrict our attention to tagged triangulations $T$ of the following kind.
		
		\begin{definition}\label{def:idealize}
			Let $T^\circ$ be an ideal triangulation of $\dpoly$ or $\apoly$. We define the \emph{corresponding tagged triangulation} $T$ as follows. If $T^\circ$ has no loops, let $T$ be the tagged triangulation with the same arcs, all tagged plain. If $T^\circ$ has a loop $\lambda$ enclosing a radius $\rho$, then let $T$ be the tagged triangulation replacing $\lambda$ with $\rho^{\bowtie}$ and all other arcs the same as in $T^{\circ}$.
			We say that tagged triangulations arising in this way \emph{correspond to an ideal triangulation}.
		\end{definition}
		
		\begin{rmk} \label{rmk:idealOnly}
			If $T$ is a tagged triangulation of $\dpoly$ that does not correspond to an ideal triangulation, then $T$ has all radii notched. This means $T^p$ does correspond to an ideal triangulation. By Proposition~\ref{prop:flip}, each $L_{T, \gamma}$ can be obtained from $L_{\gamma^p, T^p}$ by a simple change of variables. Thus, we need only provide formulas for $L_{T, \gamma}$ where $T$ corresponds to an ideal triangulation.
		\end{rmk}
		
		Note that all tagged triangulations of $\apoly$ correspond to ideal triangulations.
		
		\subsubsection{Expansion formulas for plain arcs}
		
		We now have all of the ingredients to give $L_{T,\gamma}$ for $\gamma$ a plain tagged arc and $T$ a tagged triangulation corresponding to an ideal triangulation. 
		
		In what follows, if $\lambda$ is a loop enclosing a radius $\rho$, we set $x_\lambda:=x_\rho x_{\rho^{\bowtie}}$. We also set $L_{T,\lambda}:=L_{T,\rho}L_{T,\rho^{\bowtie}}$.

		\begin{thm}\label{thm:unnotchedExpansion}\cite[Theorem 4.10]{MSW} 
			Let $(S, M)$ be a polygon or punctured polygon with boundary segments $\{\zeta_1, \dots, \zeta_r\}$, and let $T^\circ=\{\tau_1, \dots, \tau_n\}$ be an ideal triangulation with corresponding tagged triangulation $T$. Consider an oriented ordinary arc $\gamma$ (which may be a loop), and let $\tau_{i_1},...,\tau_{i_d}$ be the sequence of arcs $\gamma$ intersects in $T^\circ$. Then
			
			\begin{equation} \label{eq:unnotchedExpansion}
				L^{\text{bd}}_{T, \gamma}=\frac{1}{x_{\tau_{i_1}}...x_{\tau_{i_d}}}\sum_{M} \wt^{\text{bd}}(M),\quad
				L^{\text{pc}}_{T, \gamma}=\frac{1}{w_{\tau_{i_1}}...w_{\tau_{i_d}}}\sum_{M} \wt^{\text{pc}}(M),\quad
				L^{\text{nf}}_{T, \gamma}=\frac{1}{z_{\tau_{i_1}}...z_{\tau_{i_d}}}\sum_{M} \wt^{\text{nf}}(M)
			\end{equation}
			
			where the sum is over perfect matchings $M$ of the snake graph $G_{T^\circ, \gamma}$.
		\end{thm}
		
		Note that Theorem~\ref{thm:unnotchedExpansion} gives expansion formulas for all cluster variables in a type $A$ cluster algebra.

		\begin{rmk} \label{rmk:easyNotchedExpansions}
			There are two situations in which Theorem~\ref{thm:unnotchedExpansion} gives $L_{T,\rho^{\bowtie}}$ for a notched radius $\rho$. The first: if $T=T^p$, then $L_{T,\rho^{\bowtie}}$ for any radius $\rho$ can be obtained from $L_{T,\rho}$ using Proposition~\ref{prop:flip}. The second: if $\rho \in T$, then let $\lambda$ be the loop enclosing $\rho$. Then $L_{T,\lambda}=L_{T,\rho}L_{T,\rho^{\bowtie}}$, and since $L_{T,\rho}=x_\rho$ (resp. $w_\rho$ or $z_\rho$) is a cluster variable by assumption we have $L_{T,\rho^{\bowtie}}= (1/x_\rho) L_{T,\lambda}$ (resp. $w_\rho$ or $z_\rho$).
		\end{rmk}
		
		\begin{example}
			For $\gamma$ and $T$ with $G_{T, \gamma}$ as in Figure \ref{fig:expansion-ex}, we have the following expansion formulas (note how each term corresponds to a matching):
			$$L_{T,\gamma}^{\text{bd}}=\frac{x_{\tau_{1}}x_{\tau_{2}}x_{\delta_{1}}x_{\delta_{5}}+
				x_{\tau_{1}}x_{\tau_{2}}x_{\tau_{3}}x_{\tau_{5}}+
				x_{\tau_{1}}x_{\tau_{4}}x_{\delta_{1}}x_{\delta_{6}}+
				x_{\tau_{3}}x_{\tau_{4}}x_{\delta_{1}}x_{\delta_{7}}}{x_{\tau_{2}}x_{\tau_{3}}x_{\tau_{4}}}$$
			$$L_{T,\gamma}^{\text{nf}}=\frac{z_{\tau_{1}}z_{\tau_{2}}+
				z_{\tau_{1}}z_{\tau_{2}}z_{\tau_{3}}z_{\tau_{5}}+
				z_{\tau_{1}}z_{\tau_{4}}+
				z_{\tau_{3}}z_{\tau_{4}}}{z_{\tau_{2}}z_{\tau_{3}}z_{\tau_{4}}}$$
			$$L_{T,\gamma}^{\text{pc}}=\frac{w_{\tau_{1}}w_{\tau_{2}}y_{\tau_{4}}+
				w_{\tau_{1}}w_{\tau_{2}}w_{\tau_{3}}x_{\tau_{5}}+
				w_{\tau_{1}}w_{\tau_{4}}y_{\tau_{3}}y_{\tau_{4}}+
				w_{\tau_{3}}w_{\tau_{4}}y_{\tau_{2}}y_{\tau_{3}}y_{\tau_{4}}}{w_{\tau_{2}}w_{\tau_{3}}w_{\tau_{4}}}.$$
		\end{example}
		\begin{figure}
			\centering
			\includegraphics[width=0.9\linewidth]{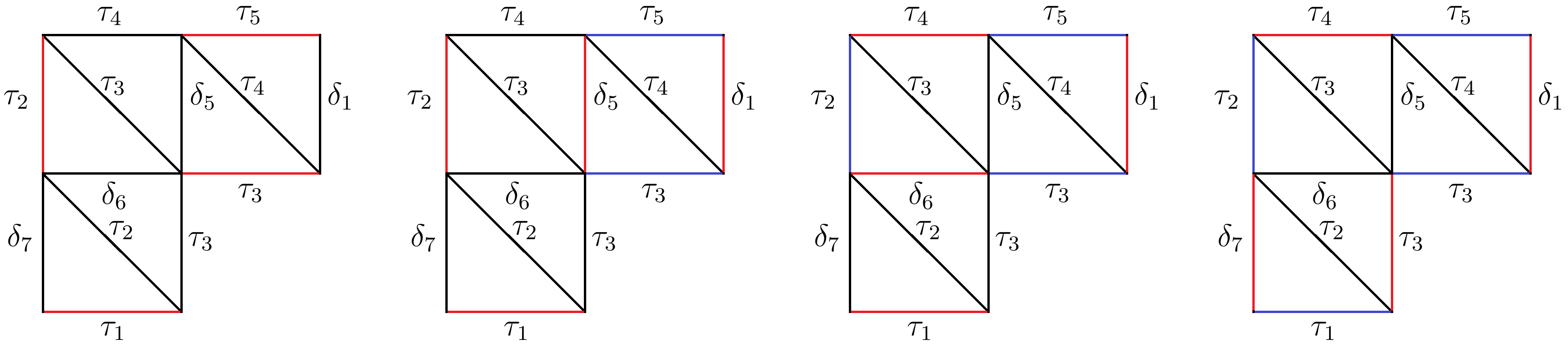}
			\caption{All matchings of a snake graph $G_{T,\gamma}$ (red), with edges of the bottom matching in blue.}
			\label{fig:expansion-ex}
		\end{figure}
		
		\subsubsection{Expansion formulas for notched arcs} \label{subscn:notchedArcExpansion}
		Now we consider expansion formulas for $x_{\rho^{\bowtie}}$ (resp. $w_{\rho^{\bowtie}}$ or $z_{\rho^{\bowtie}}$), for $\rho$ a radius. They are of a similar flavor, but involve $\rho$-symmetric matchings of snake graphs.
		
		Fix $T^\circ=\{\tau_1, \dots, \tau_n\}$ an ideal triangulation of $\dpoly$ and let $T$ be the corresponding tagged triangulation.  We may assume that $\rho \notin T$ and that $T \neq T^p$ (see Remark~\ref{rmk:easyNotchedExpansions}), so in fact $T=T^\circ$. Let $\lambda$ be the loop enclosing $\rho$.
		
		\begin{defn}
			The snake graph $G_{T, \lambda}$ contains two disjoint subgraphs isomorphic to $G_{T, \rho}$ as labeled graphs, one on each end. We denote these graphs by $G_{T,\rho,1}$ and $G_{T,\rho,2}$. Let $H_{T,\rho,i}$ be the subgraph of $G_{T,\rho,i}$ obtained by deleting edges labeled by radii. (See Figure~\ref{fig:endsEx})
		\end{defn}
		
		\begin{figure}
			\centering
			\includegraphics[width=0.9\linewidth]{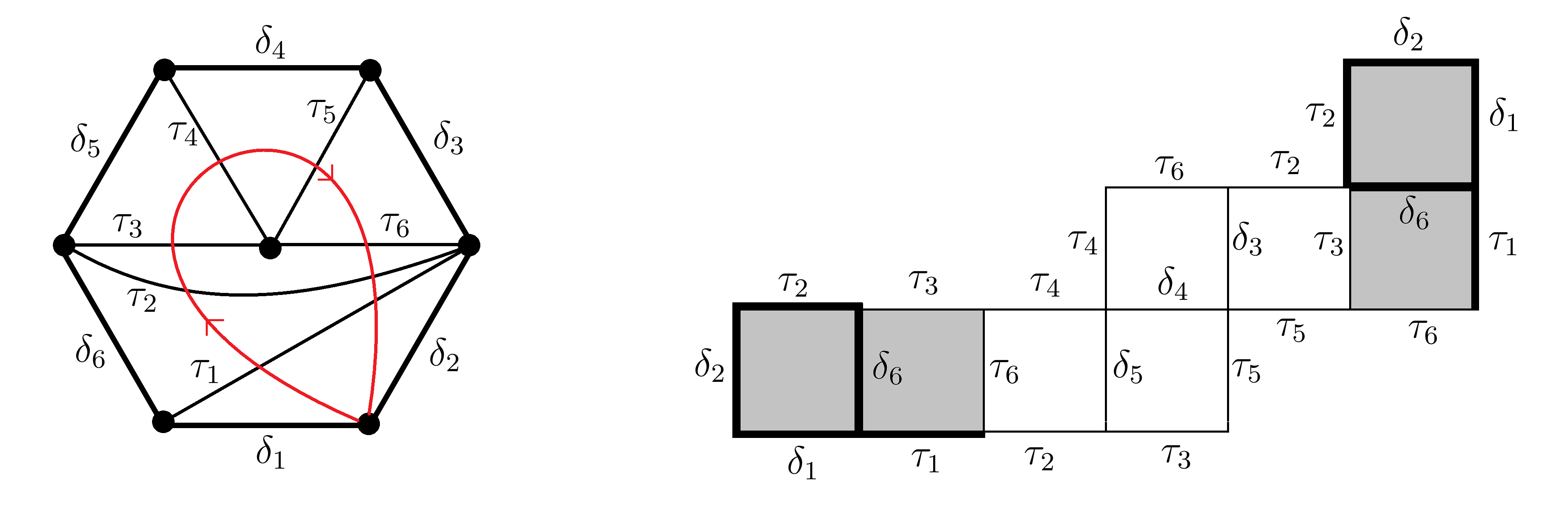}
			\caption{\label{fig:endsEx} An example of $G_{T, \lambda}$ for $\lambda$ a loop (shown in red). The squares of the subgraphs $G_{T,\rho,i}$ are shaded; the edges of $H_{T,\rho,i}$ are in bold.}
		\end{figure}
		
		\begin{defn}
			A perfect matching $M$ of $G_{T, \lambda}$ is \emph{$\rho$-symmetric} if $M|_{H_{T,\rho,1}} \cong M|_{H_{T,\rho,2}}$. The \emph{weight} of a $\rho$-symmetric matching $M$ is given by
			\[\owt(M):= \frac{\wt(M)}{\wt(M|_{G_{T, \rho, i}})}\]
			where $i$ is chosen so that $M|_{G_{T, \rho, i}}$ is a perfect matching of $G_{T, \rho, i}$ (such an $i$ exists by \cite[Lemma 12.4]{MSW}).
			
			The weight vector of a $\rho$-symmetric matching is the exponent vector of its weight.
		\end{defn}
		\begin{figure}
			\begin{minipage}{0.5\linewidth}
				\centering
				\includegraphics[width=0.7\linewidth]{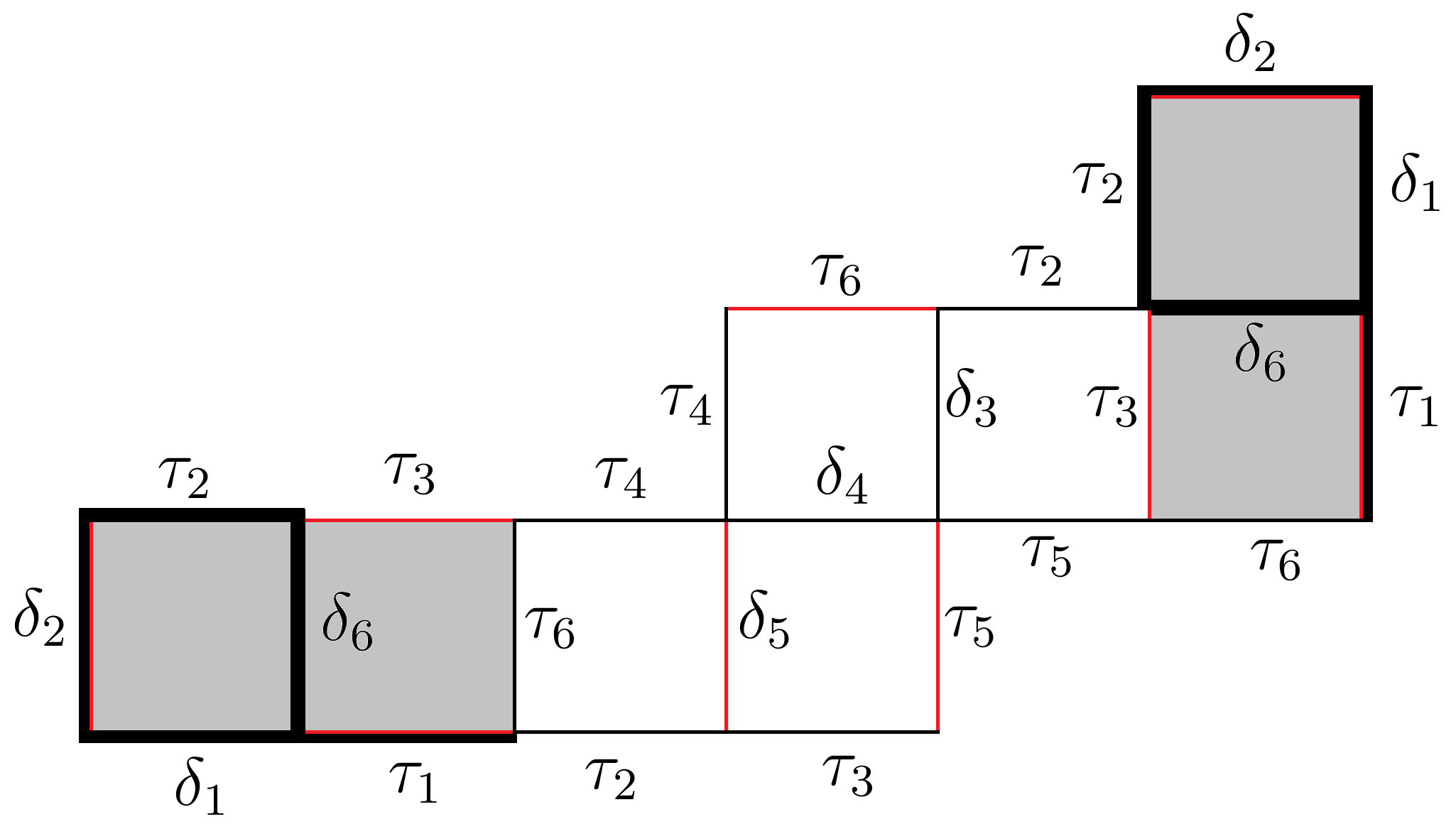}
			\end{minipage}
			\begin{minipage}{0.48\linewidth}
				\centering
				\includegraphics[width=0.7\linewidth]{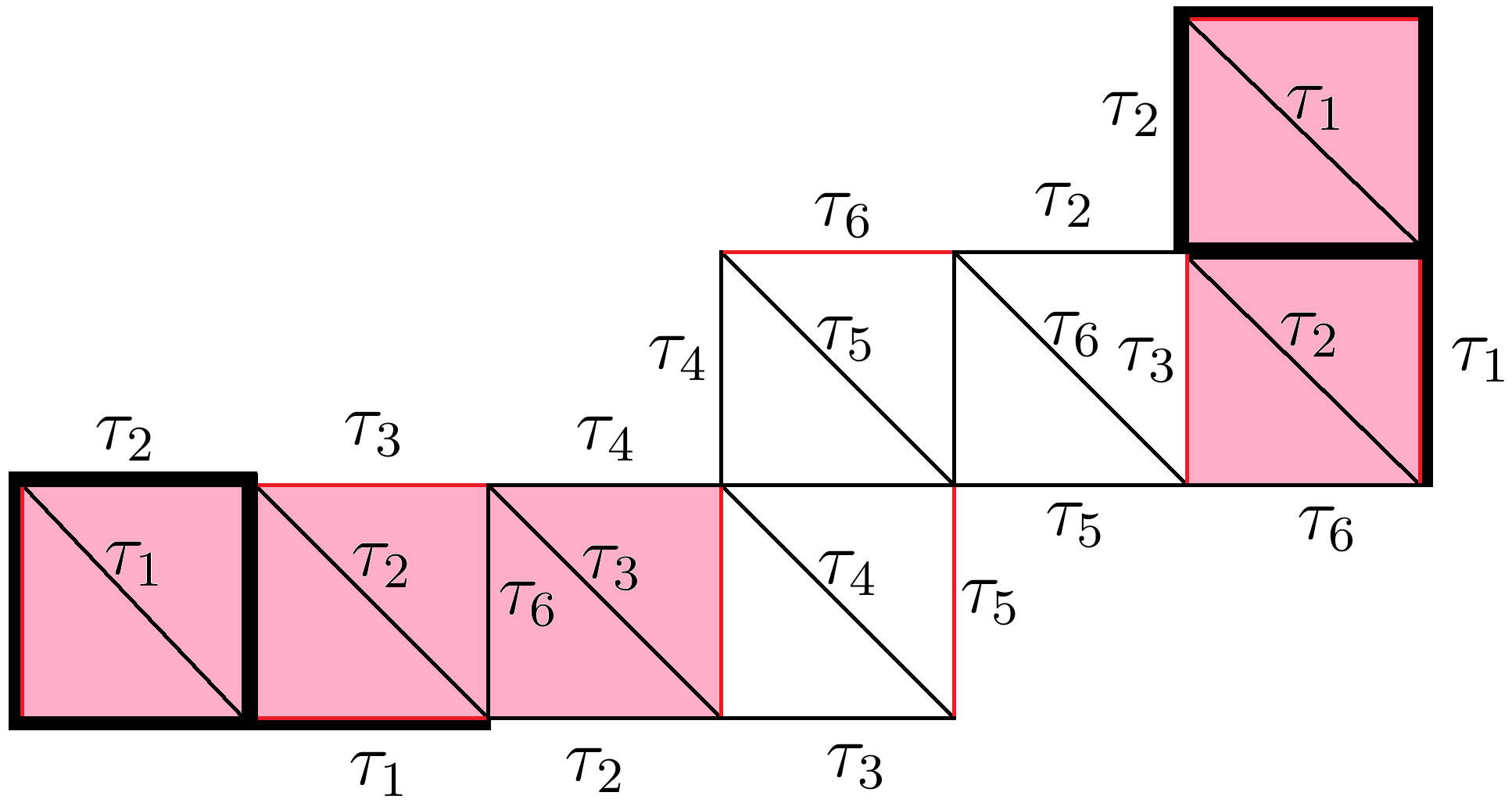}
			\end{minipage}
			\caption{A $\rho$-symmetric matching $M$ (red) of snake graph $G_{T,\lambda}$ from Figure \ref{fig:endsEx} with boundary frozen variables, and of the corresponding snake graph with principal coefficients. The pink squares correspond to $C(M)$.}
			\label{fig:wt-bar-ex}
		\end{figure}
		\begin{example}
			For the matching $M$ in Figure \ref{fig:wt-bar-ex}, we have the weights:
			$$\overline{\text{wt}}^{\text{bd}}(M)=
			\frac{x_{\tau_{1}}^2x_{\tau_{3}}^2x_{\tau_{5}}x_{\tau_{6}}x_{\delta_{2}}^2x_{\delta_{5}}}{x_{\tau_{1}}x_{\tau_{3}}x_{\delta_{2}}}=
			x_{\tau_{1}}x_{\tau_{3}}x_{\tau_{5}}x_{\tau_{6}}x_{\delta_{2}}x_{\delta_{5}}$$
			$$\overline{\text{wt}}^{\text{pc}}(M)=
			\frac{w_{\tau_{1}}^2w_{\tau_{3}}^2w_{\tau_{5}}w_{\tau_{6}}y_{\tau_{1}}^2y_{\tau_{2}}^2}{w_{\tau_{1}}w_{\tau_{3}}y_{\tau_{1}}y_{\tau_{2}}}=
			w_{\tau_{1}}w_{\tau_{3}}w_{\tau_{5}}w_{\tau_{6}}y_{\tau_{1}}y_{\tau_{2}}$$
		\end{example}
		
		\begin{thm}\label{thm:notchedExpansion}\cite[Theorem 4.17]{MSW} 
			Let $T=\{\tau_1, \dots, \tau_n\}$ be a tagged triangulation of $\dpoly$ which is also an ideal triangulation. Suppose $\rho \notin T$ is a radius, and let $\lambda$ be the loop enclosing $\rho$. Let $\tau_{i_1},...,\tau_{i_k}$ be the sequence of arcs of $T$ that $\lambda$ intersects, and suppose $\rho$ intersects  $\tau_{i_1},...,\tau_{i_d}$. Then
			\begin{itemize}
				\item $L^{\text{bd}}_{T, \rho^{\bowtie}}= \frac{1}{x_{\tau_{i_{d+1}}}\cdots x_{\tau_{i_k}}} \sum_M \owt^{\text{bd}}(M)$
				\item $L^{\text{pc}}_{T, \rho^{\bowtie}}= \frac{1}{w_{\tau_{i_{d+1}}}\cdots w_{\tau_{i_k}}} \sum_M \owt^{\text{pc}}(M)$
				\item $L^{\text{nf}}_{T, \rho^{\bowtie}}= \frac{1}{z_{\tau_{i_{d+1}}}\cdots z_{\tau_{i_k}}} \sum_M \owt^{\text{nf}}(M)$
			\end{itemize} 
			where the sum is over $\rho$-symmetric matchings of $G_{T, \lambda}$.
		\end{thm}	
	\section{Newton polytopes and perfect matching polytopes}
			In this section, we define perfect matching polytopes and give their relation to the Newton polytopes $N(L_{T,\gamma})\subset \rr^T$. 
		
		To simplify notation, set $N(T, \gamma):=N(L_{T,\gamma})$, with $N^{\text{bd}}(T, \gamma)$, $N^{\text{pc}}(T, \gamma)$, and $N^{\text{nf}}(T, \gamma)$ defined accordingly.
		
		\begin{rmk} 
			Proposition~\ref{prop:flip}, particularly the substitution in \eqref{eq:flip}, implies $N(T, \gamma)$ and $N(T^p, \gamma^p)$ differ only by renaming coordinates. Thus, when proving the saturation of $N(T, \gamma)$, we may assume the tagged triangulation $T$ corresponds to an ideal triangulation. We will make this assumption for the remainder of the paper.
		\end{rmk}
		
		Recall the definition of weight vector $w^M$ of a perfect matching $M$ (see Definition~\ref{def:matching}).
		
		\begin{defn}
			Given a snake graph $G_{T,\gamma}$, the \emph{perfect matching polytope} $P(G)$ is the convex hull of $\{w^M: M \text{ is a perfect matching of }G\}$. We say that the perfect matching polytope $P(G)$ is \emph{saturated} if every lattice point in $P(G)$ is the weight vector of a matching, and we say it is \emph{empty} if every lattice point in $P(G)$ is a vertex.  
		\end{defn}
		Note that $P^{\text{bd}}(G)$, $P^{\text{pc}}(G)$, and $P^{\text{nf}}(G)$ are defined using $w_{\text{bd}}^{M}$, $w_{\text{pc}}^{M}$, and $w_{\text{nf}}^{M}$ respectively. Each polytope lies in the vector space containing the corresponding weight vectors, as given in Definition \ref{def:matching}.
		\begin{defn}
			Let $G_{T,\gamma}$ be a snake graph, and let $G'$ be a snake graph identical to $G$ but with edges given unique labels. The \emph{lifted perfect matching polytope} $\overline{P}(G)$ is the convex hull of the indicator vectors $\chi_M\in\rr^{E(G)}$, where $M$ ranges over all perfect matchings of $G'$. 
			
			Similarly, the \emph{lifted principal perfect matching polytope} $\overline{P}^{pc}(G)$ is the convex hull of the indicator vectors $(\chi_M,\chi_{C(M)})\in\rr^{E(G)}\times\rr^{S(G)}$, where $S(G)$ is the set of squares of $G'$ and $M$ ranges over all perfect matchings of $G'$. 
		\end{defn}
		The polytope $\overline{P}(G)$ is frequently called the perfect matching polytope; we depart from this because of our interest in graphs with non-distinct edge labels. 
		
		The following lemma explains our interest in perfect matching polytopes.
		
		\begin{lemma} \label{lem:saturationEquiv}
			Let $T^\circ=\{\tau_1, \dots, \tau_n\}$ be an ideal triangulation of $\apoly$ or $\dpoly$, and let $T$ be the corresponding tagged triangulation. Let $\gamma$ be an ordinary arc (including possibly a loop). Then $N(T, \gamma)$ is saturated (resp. empty) if and only if $P(G_{T^\circ, \gamma})$ is saturated (resp. empty).
		\end{lemma}
		
		\begin{proof}
			We use notation for $\mca^{\text{bd}}$, as the proofs for the other two are identical. 
			
			There are two cases. If $T^\circ=T$ (i.e. $T^\circ$ has no loops), then by Theorem~\ref{thm:unnotchedExpansion}, the set of weight vectors $W:=\{w^M: M \text{ is a perfect matching of }G_{T^\circ, \gamma}\}$ differs from the support of $L_{T, \gamma}$ by an element of $\zz^T$. So $P(G_{T^\circ, \gamma})$ is an integer translate of $N(T, \gamma)$. 
			
			If $T^\circ \neq T$, then $T^\circ$ contains a loop $\lambda$ enclosing a radius $\rho$. The tagged triangulation $T$ contains $\rho^{\bowtie}$ and not $\lambda$, and in \eqref{eq:unnotchedExpansion}, we set $x_\lambda=x_\rho x_{\rho^{\bowtie}}$ to obtain a term of $L_{T, \gamma}$ from the weight of a matching of $G_{T^\circ, \gamma}$.  
			
			Consider the map $\xi: \rr^{T^\circ} \mapsto \rr^T$, which fixes coordinates indexed by $T^\circ \cap T$ and sends the others to
			\begin{align*}
				\xi(v)_\rho&= v_\lambda + v_\rho\\
				\xi(v)_{\rho^{\bowtie}}&=v_\lambda.
			\end{align*}
			Theorem~\ref{thm:unnotchedExpansion} implies that there exists $\eta\in \zz^T$ such that $\xi(W)+\eta$ is the support of $L_{T, \gamma}$.
			In particular, $\xi(P(G_{T^\circ, \gamma}))+\eta$ is equal to $N(T, \gamma)$. Since $\xi$ and translation by $\eta$ are both bijective on lattice points, the statement follows.
			
		\end{proof}
		
		The lifted (principal) perfect matching polytopes are our main tool to prove saturation of perfect matching polytopes. The lifted matching polytopes are combinatorially easy to understand. We then carefully analyze the effects of projecting the lifted matching polytopes to the matching polytopes, using the following natural projection.
		
		Let $G=G_{T,\gamma}$ and let $Z$ denote as usual the boundary arcs of the surface $(S, M)$. We have projection maps
		\begin{align}
			\pi^{\text{bd}}:\mathbb{R}^{E(G)} &\to \mathbb{R}^T \times \mathbb{R}^Z\\
			\pi^{\text{nf}}:	\mathbb{R}^{E(G)} &\to \mathbb{R}^T\\
			\pi^{\text{pc}}:\mathbb{R}^{E(G) \times S(G)}& \to \mathbb{R}^T \times \mathbb{R}^T
		\end{align}
		given by summing coordinates according to the labels of the indexing edge or square. For example, the $\tau_j$th coordinate of $\pi^{\text{bd}}(v)$ is 
		\[\sum_{\substack{{e \in E(G):}\\{\ell(e)=\tau_j}}} v_e.\]
		
		We will simply write $\pi$ if we do not want to specify the vector spaces, or if it is clear from context. 
		
		Note that $\pi$ is a surjection from incidence vectors of $\overline{P}(G)$ (resp. $\overline{P}^{\text{pc}}(G)$) to weight vectors of $P^{\text{bd}}(G)$ and $P^{\text{nf}}(G)$ (resp. $P^{\text{pc}}(G)$).
		
		The lifted perfect matching polytopes have a straightforward description.
		
		\begin{prop}(\cite{matchingPoly}, see also \cite[Theorem 7.3.4]{matchingPoly2})\label{PM(G)}
			Let $G$ be a snake graph. Then
			\[\overline{P}(G)= \{w \in\rr^{E(G)}: w_e \geq 0, \sum_{e\ni v}w_e=1 \text{ for all }  v \in G\}.\]
		\end{prop}
		\begin{cor}\label{PMPC}
			Let $G$ be a snake graph. Then for each $s\in S(G)$, there is some $e_s\in E(G)$, and a map $f_s$, which is either $x\mapsto x$ or $x\mapsto 1-x$, such that
			$$\overline{P}^{\text{pc}}(G)=
			\{(w_1,w_2)\in\rr^{E(G)}\times\rr^{S(G)}: ({w_1})_e \geq 0, ({w_2})_s=f_s(({w_1})_{e_s}),\sum_{e\ni v}w_e=1 \text{ for all }  v \in G\} $$
		\end{cor}
		\begin{proof}
			Each square $s\in S(G)$ has an edge $e_s$ on the boundary of $G$. If $e_s\notin M_0$, then for any matching $M$ we have $e_s\in M$ if and only if $s\in C(M)$; if $e_s\in M_0$, then for any matching $M$ we have $e_s\in M$ if and only if $s\notin C(M)$. In the former case, let $f_s:x\mapsto x$ and in the latter case let $f_s=x\mapsto 1-x$. If $(w_1, w_2)=\chi(M)$ is a principal incidence vector for a matching $M$, then we have $({w_2})_s=f_s(({w_1})_{e_s})$. The same relation must hold in any element of the convex hull of the incidence vectors. The relations given in Lemma \ref{PM(G)} must also hold. And since we have a (coordinate) projection $\overline{P}^{\text{pc}}(G)\to \overline{P}(G)$, the dimension of the former must be at least that of the latter, so there cannot be any additional relations. 
		\end{proof}
		The preceding proof shows that we have an affine map $\alpha: \overline{P}^{\text{pc}}(G)\to \overline{P}(G)$, which is inverse to the natural projection $\overline{P}(G)\to \overline{P}^{\text{pc}}(G)$. Coordinate-wise, $\alpha$ is defined by $(\alpha(v)_1)_e=v_e$ and $(\alpha(v)_2)_s=f_s(v_s)$.
		
		Next, we show that lifted matching polytopes have the nice properties we care about.	
		
		\begin{lemma}\label{lem:G-sat}
			Let $G$ be a snake graph. Then $\overline{P}(G)$ and $\overline{P}^{\text{pc}}(G)$ are empty and, in particular, are saturated.
		\end{lemma}
		\begin{proof}
			By Lemma \ref{PM(G)}, each lattice point $\eta$ in $\overline{P}(G)$ has $\eta_{e}=0$ or $\eta_{e}=1$. Further, for each vertex $v$ of $G$, there is exactly one edge $e$ containing $v$ such that $\eta_e=1$. Thus, each lattice point is the weight vector of a perfect matching for $G$. 
			
			Note that all weight vectors are $0/1$ vectors with $|E(G)|/2$ coordinates equal to $1$. This implies that a nontrivial convex combination of weight vectors is not a weight vector, so all weight vectors are vertices of $\overline{P}(G)$. 
			
			Since the affine isomorphism between $\overline{P}(G)$ and $\overline{P}^{\text{pc}}(G)$ in the proof of Corollary \ref{PMPC} maps lattice points to lattice points, the desired statement holds for $\overline{P}^{\text{pc}}(G)$ as well.
		\end{proof}
		
		\begin{rmk}\label{rmk:preimage}
			By Lemma~\ref{lem:G-sat}, a lattice point $v \in P(G)$ is a weight vector if and only if $\pi^{-1}(v) \subset \overline{P}(G)$ (resp. $\overline{P}^{\text{pc}}(G)$) contains a lattice point. This will be our general strategy for proving saturation of $P(G)$. For emptiness, we will consider collections of weight vectors such that some convex combination of them is equal to a lattice point, and then show that at most one vector in the combination has a positive coefficient. 		
		\end{rmk}
	\section{Type A cluster variable Newton polytopes}
		Here we restrict our attention to cluster variables of cluster algebras of type $A$; that is, to triangulations $T$ and arcs $\gamma$ of a polygon $\apoly$. 
		\subsection{Boundary frozen variables}
		We first show that the projection map $\pi^{\text{bd}}$ is very well-behaved on $\overline{P}(G_{T,\gamma})$. For the rest of this section, let $\pi=\pi^{\text{bd}}$ and $P(G)=P^{\text{bd}}(G)$. 
		
		\begin{lemma}\label{map-pi}
			Let $T$ be a triangulation of $\apoly$ and let $\gamma$ be an arc. For $v$ a lattice point of $P(G_{T,\gamma})$, $\pi^{-1}(v)\subseteq\zz^{E(G_{T,\gamma})}$. 
		\end{lemma}
		\begin{proof}

			We will employ induction on the number of squares in $G_{T, \gamma}$. If there is exactly one square, then all edges will have different labels, so $\pi$ is the identity map. 
			
			If $G_{T,\gamma}$ has more than one square, assume the lemma holds for all smaller snake graphs, and suppose for the sake of contradiction that there exists a non-lattice point $\overline{v}\in \overline{P}(G_{T,\gamma})$ such that $\pi(\overline{v})\in P(G_{T,\gamma})$ is a lattice point. Let the boundary segments and arcs of $T$ around one endpoint of $\gamma$ be as shown in Figure \ref{diagrams} on the left, so the final tile of $G_{T, \gamma}$ is as in Figure \ref{diagrams} on the right. Let $\gamma'$ be the arc obtained by changing the endpoint of $\gamma$ to vertex $W$ in Figure \ref{diagrams} on the left, so $G_{T,\gamma'}$ is equal to $G_{T,\gamma}$ with the last square removed. Note that $a,b,c$ each appear only once as edge labels of $G_{T,\gamma}$, so  $\overline{v}_a,\overline{v}_b,\overline{v}_c\in\{0,1\}$. And $\overline{v}_a+\overline{v}_d=1$, so $\overline{v}_d\in\{0,1\}$ too. 
			\begin{figure}
				\centering
				\begin{minipage}{0.55\linewidth}
					\includegraphics[width=0.7\linewidth]{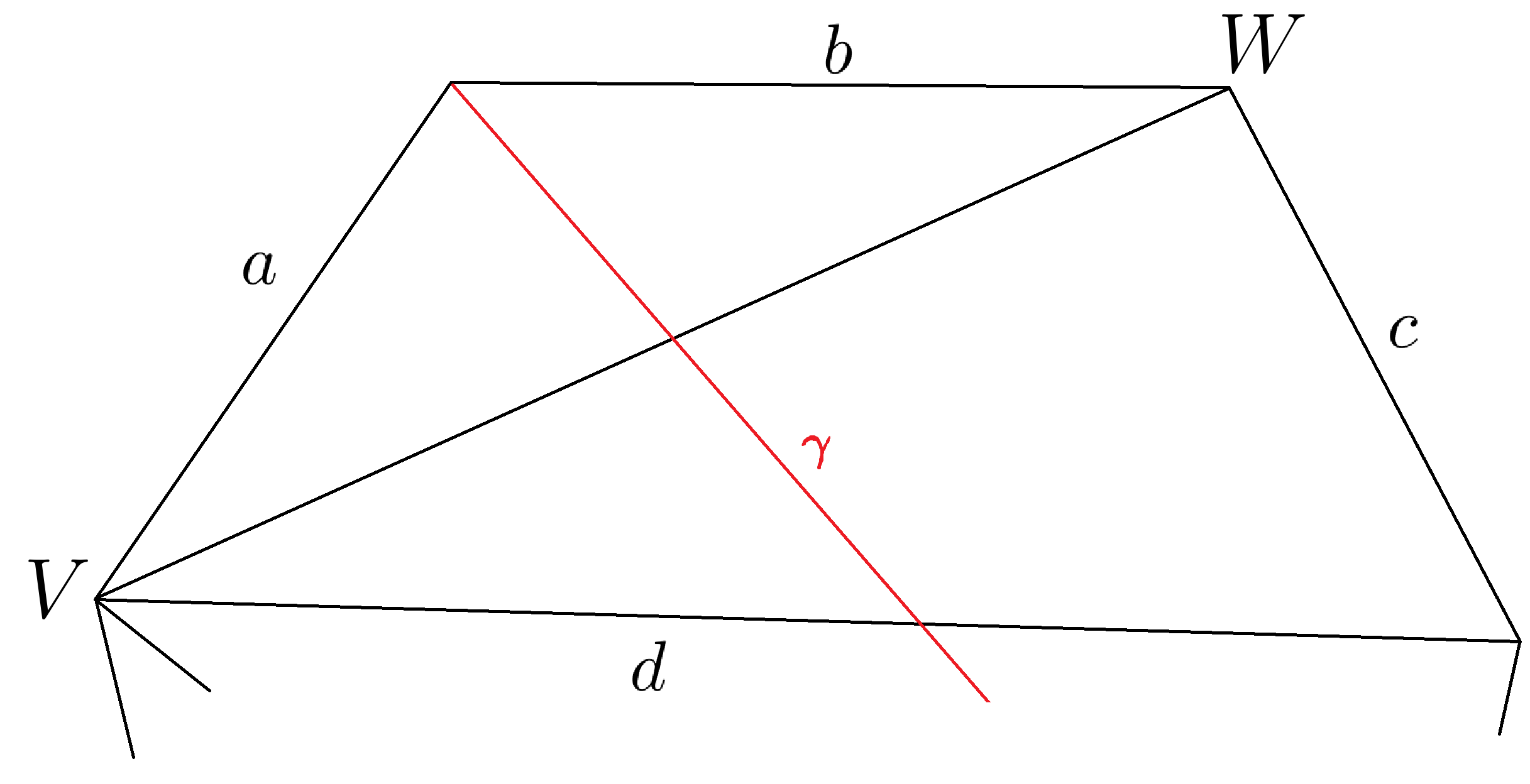}
				\end{minipage}
				\begin{minipage}{0.35\linewidth}
					\includegraphics[width=0.5\linewidth]{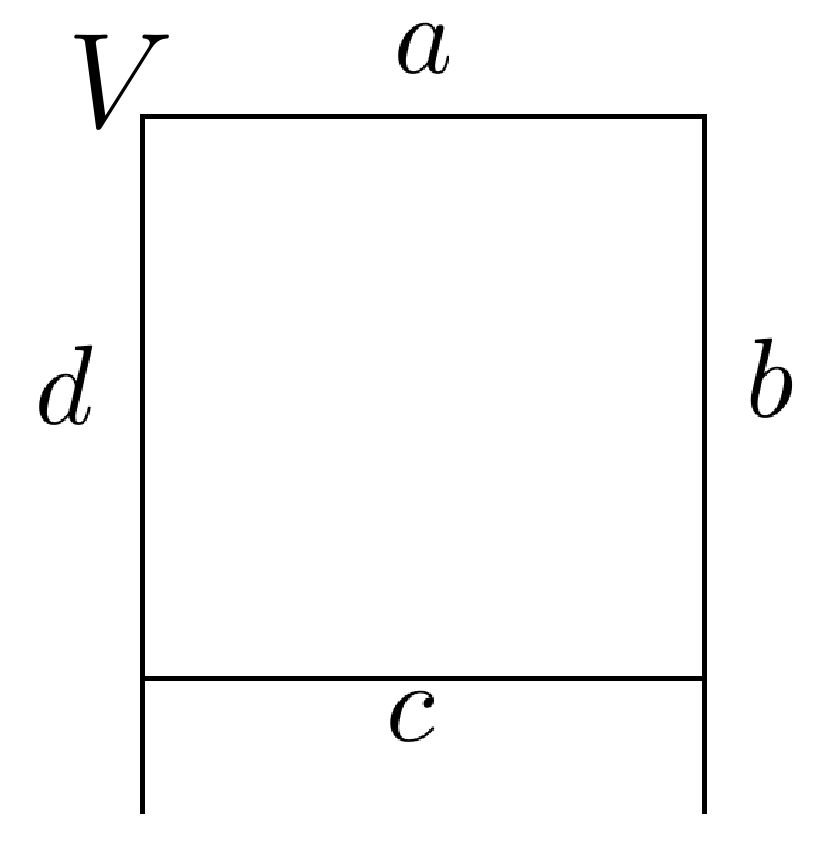}
				\end{minipage}
				\caption{On the left, the final quadrilateral of $T$ crossed by arc $\gamma$. On the right, the final square of $G_{T, \gamma}$.}
				\label{diagrams}
			\end{figure}
			
			If $\overline{v}_a=1$, then $\overline{v}_b=\overline{v}_d=0$ and $\overline{v}_c=1$, so some other coordinate of $\overline{v}$ must be a non-integer. Deleting coordinates $\overline{v}_a,\overline{v}_b,\overline{v}_d$ yields a non-lattice point $\overline{v}'$ in $\overline{P}(G_{T,\gamma'})$ with $\pi(\overline{v}')\in P(G_{T,\gamma'})$ a lattice point, contradicting the inductive hypothesis. If instead $\overline{v}_a=0$, define $\overline{v}'$ by $\overline{v}'_c=\overline{v}'_a=1, \overline{v}'_b=\overline{v}'_d=0$ and all other coordinates equal to those of $\overline{v}$; deleting coordinates $\overline{v}'_a,\overline{v}'_b,\overline{v}'_d$ yields a non-lattice point of $\overline{P}(G_{T,\gamma'})$ similarly contradicting the inductive hypothesis.
		\end{proof}
		Now we can prove our first main result.
		\begin{theorem}\label{AB-saturated} Let $\mca^{\text{bd}}$ be a type $A$ cluster algebra with boundary coefficients. Then the Newton polytope of any cluster variable, written as a Laurent polynomial in an arbitrary seed, is saturated and empty.
		\end{theorem}
		\begin{proof}Let $T$ be a triangulation of $\apoly$ and $\gamma$ an arc.
			By Lemma~\ref{lem:saturationEquiv}, $N^{\text{bd}}(T,\gamma)$ is saturated if and only if $P(G_{T,\gamma})$ is saturated. By Lemma~\ref{map-pi}, every lattice point $v$ in $P(G_{T,\gamma})$ has an integer point $\overline{v} \in \pi^{-1}(w)$. By Remark~\ref{rmk:preimage}, this implies $P(G_{T,\gamma})$ is saturated.
			
			Lemma \ref{map-pi} and Lemma \ref{lem:G-sat} together imply that each lattice point of $P(G_{T,\gamma})$ is the image under $\pi$ of a vertex of $\overline{P}(G_{T,\gamma})$. By \cite[proof of Theorem 4.13]{Kalman}, $\pi$ maps vertices of $\overline{P}(G_{T,\gamma})$ to vertices of $P(G_{T,\gamma})$. \footnote{Note that \cite[Theorem 4.13]{Kalman} does not imply this theorem. Take $P_1$ to be the line segment in $\rr^2$ with endpoints (0, 0) and (0, 2) and $P_2$ the line segment in $\rr$ with endpoints 0 and 2. Projecting onto the 2nd coordinate takes $P_1$ to $P_2$, with vertices going to vertices, but the lattice point $1 \in P_2$ does not have a lattice point in its preimage.}
		\end{proof}
		\begin{remark}
			Note the key properties of triangulations of $\apoly$ we used in Lemma \ref{map-pi}: that each quadrilateral in a triangulation involves four distinct arcs, and that there are strict limits on the number of times each edge label occurs in $G_{T,\gamma}$. The former will fail for punctured surfaces, where we encounter self-folded triangles, and the latter will fail for surfaces with non-finite cluster type, where arcs can intersect arbitrarily many times. 
		\end{remark}
		\subsection{Principal coefficients}
		For the case of type $A$ cluster algebras with principal coefficients, we will use a similar strategy to the case of boundary frozen variables. However, the additional structure of principal coefficients will simplify our work. For this section, let $\pi=\pi^{\text{pc}}$.
		\begin{lemma}\label{options}
			For any arc $\tau_i$ labeling an edge in $G_{T,\gamma}$, one of the following is true for all $\overline{v}\in \overline{P}^{\text{pc}}(G_{T,\gamma})\subset\rr^T\times\rr^T$:
			\begin{itemize}
				\item There is a unique edge $e$ in $G_{T,\gamma}$ with $\ell(e)=\tau_i$.
				\item There is $\sigma\in T$ such that $(\overline{v}_1)_{\tau_i}=(\overline{v}_2)_{\sigma}$.
				\item There is $\sigma\in T$ such that $(\overline{v}_1)_{\tau_i}=1-(\overline{v}_2)_{\sigma}$.
			\end{itemize}
		\end{lemma}
		\begin{proof}
			%
			
			If the first condition does not hold, then $\tau_i$ appears exactly twice in $G_{T,\gamma}$ and labels edges on the boundary of $G_{T,\gamma}$. Let $\sigma$ be the label of a square in $G_{T,\gamma}$ with one of its boundary edges labeled by $\tau_i$. Then by Corollary \ref{PMPC} and its proof, either $(\overline{v}_1)_{\tau_i}=(\overline{v}_2)_{\sigma}$ or $(\overline{v}_1)_{\tau_i}=1-(\overline{v}_2)_{\sigma}$.
			
		\end{proof}
		\begin{lemma}\label{P-map-pi}
			$\pi:\overline{P}^{\text{pc}}(G_{T,\gamma})\to P^{\text{pc}}(G_{T,\gamma})$ maps non-lattice points to non-lattice points. 
		\end{lemma}
		\begin{proof}
			
			Let $\overline{v}$ be a point in $\overline{P}^{\text{pc}}(G_{T,\gamma})$ such that $v=\pi(\overline{v})$ is a lattice point. Since $\gamma$ intersects each arc of $T$ at most once  $(v_2)_{\sigma}=(\overline{v}_2)_{\sigma}\in\{0,1\}$ for each $\sigma\in T$. Next, for each $\tau\in T$, if there is a unique edge labeled $\tau$ then $(v_1)_{\tau}=(\overline{v}_1)_{\tau}\in\{0,1\}$. And by Lemma \ref{options}, if $\tau$ is not a unique label then $(\overline{v}_1)_{\tau}$ is either $(\overline{v}_2)_{\sigma}$ or $1-(\overline{v}_2)_{\sigma}$. Thus, all coordinates of $\overline{v}$ are integers, proving the contrapositive of the lemma.
		\end{proof}
		We can use the reasoning of this proof to show the following relationship between polytopes.
		\begin{prop}\label{lin-inv}
			There is an affine inverse $\pi^{-1}:P^{\text{pc}}(G_{T,\gamma})\to \overline{P}^{\text{pc}}(G_{T,\gamma})$ of $\pi$. In particular, $\pi$ is a combinatorial equivalence between the two polytopes.
		\end{prop}
		\begin{proof}
			As we saw in the proof of Lemma \ref{P-map-pi}, given $v=\pi(\overline{v})$, we can recover the coordinates of $\overline{v}$ by Lemma \ref{options}. We saw that each coordinate of $\overline{v}$ is either equal to a coordinate of $v$ or equal to $1$ minus a coordinate of $v$. Performing these coordinate-wise operations yields the desired map $\pi^{-1}$.
		\end{proof}
		\begin{theorem}\label{AP-saturated} Let $\mca^{\text{pc}}$ be a type $A$ cluster algebra with principal coefficients at a seed $\Sigma$. Then the Newton polytope of any cluster variable, written as a Laurent polynomial in $\Sigma$, is saturated and empty.
		\end{theorem}
		\begin{proof} Let $T$ be the triangulation corresponding to $\Sigma$ and choose an arc $\gamma$ of $\apoly$.
			As before, it suffices to prove saturation and emptiness of $P^{\text{pc}}(G_{T,\gamma})$. By Lemma \ref{P-map-pi}, the preimage under $\pi$ of any lattice point of $P^{\text{pc}}(G_{T,\gamma})$ is a lattice point of $\overline{P}^{\text{pc}}(G_{T,\gamma})$, and by Lemma \ref{lem:G-sat}, each lattice point of $\overline{P}^{\text{pc}}(G_{T,\gamma})$ is a vertex. This shows that each lattice point of $P^\text{pc}(G_{T,\gamma})$ is the weight vector of a matching, proving saturation. And since $\pi$ is a combinatorial equivalence by Proposition \ref{lin-inv}, it must map vertices to vertices, proving emptiness. 
		\end{proof}
		\subsection{No frozen variables}
		The case of Type $A$ cluster algebras with no frozen variables is considerably different from the previous cases. The added difficulty arises from the fact that the natural projection $\overline{P}(G)\to P^{\text{nf}}(G)$ is \emph{not} a combinatorial isomorphism.
		
		The following lemma restricts when non-vertex lattice points may arise.
		\begin{lemma}\label{sym-dif}
			Let $T$ be a trinagulation of $\apoly$ and $\gamma$ an arc. For $G=G_{T,\gamma}$, let $S$ be a set of matchings such that $\text{Conv}_{M\in S}\{w_{nf}^{M}\}$ contains a non-vertex lattice point. Then for any $M_1,M_2\in S$, the symmetric difference $M_1\ominus M_2$ is the union of all edges of a set of pairwise-disjoint squares of $G$.
		\end{lemma}
		\begin{proof}

			Without loss of generality assume there is a lattice point $\eta=\sum_{M\in S}c_Mw_{\text{nf}}^{M}$ with $\sum_{M\in S}c_M=1$ and $0<c_M<1$ for each $c_M$. For each coordinate $i$, either $(w_{\text{nf}}^{M})_i=(w_{\text{nf}}^{M'})_i$ for all $M,M'\in S$, or there exist $M,M'\in S$ such that $(w_{\text{nf}}^{M})_i=0$, $(w_{\text{nf}}^{M'})_i=2$. 
			
			Assume for the sake of contradiction that there exist $M_1,M_2\in S$ that violate the statement of the lemma. Let $\alpha$ be the first square in $G$ such that there exist $M_1,M_2\in S$ with $(M_1\ominus M_2)|_{\alpha}$ consisting of two or three edges. Noting that $\alpha$ cannot be the last square of $G$, let $\beta$ be the square following $\alpha$ in $G$. 
			
			We see by induction on $|G|$ that $M_1$ and $M_2$ restricted to $\alpha$ and $\beta$ must contain the edges shown in Figure \ref{fig:contradiction}: this is true if $\alpha$ is the first square in $G$, and otherwise $M_1$ and $M_2$ restricted to the first square both contain either the same edge or a pair of opposite edges; removing this edge or these edges we obtain a pair of perfect matchings for a smaller diagram. 
			
			In Figure \ref{fig:contradiction}, the edge $e$ appears in one matching but not the other. And $\ell(e)$ cannot be a boundary segment unless there are exactly two squares in the diagram (in which case the lemma is vacuously true), nor does $\ell(e)$ appear again to the northwest of $\beta$. Thus, either $(w_{\text{nf}})_{\ell(e)}\in\{0,1\}$ for all $M\in S$, or $(w_{\text{nf}})_{\ell(e)}\in\{1,2\}$ for all $M\in S$, contradicting the hypothesis and proving the lemma.
		\end{proof}
		\begin{figure}
			\centering
			\includegraphics[width=0.5\linewidth]{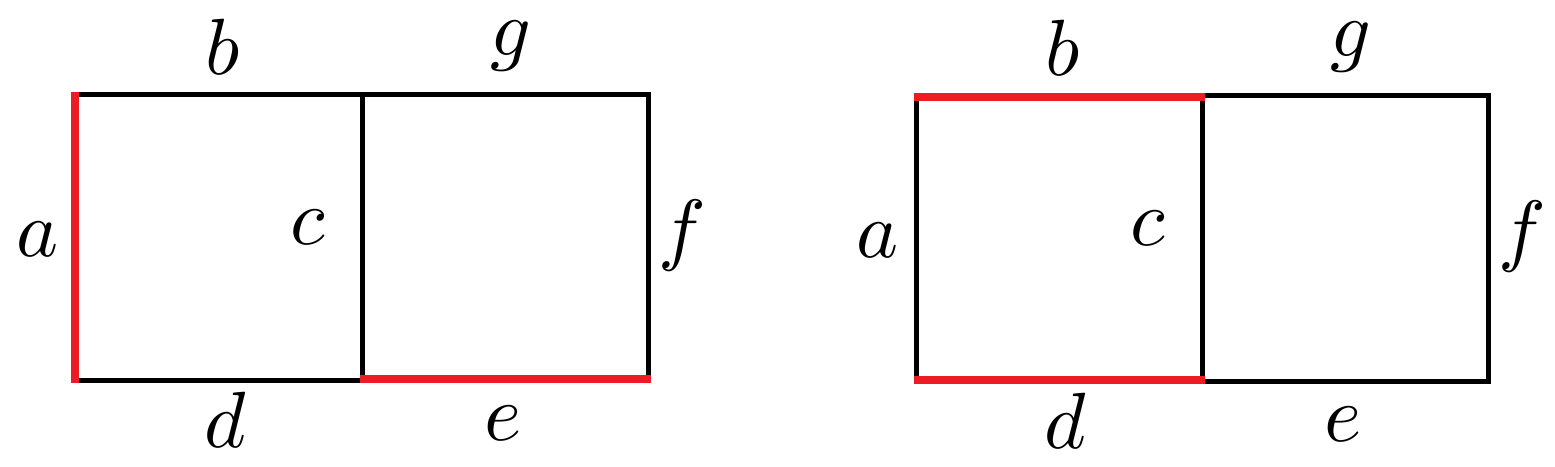}
			\caption{Necessary configuration for proof by contradiction in Lemma \ref{sym-dif}}
			\label{fig:contradiction}
		\end{figure}
		In fact, we can sharpen this lemma.
		\begin{lemma}\label{alternating}
			Assuming the same hypothesis as Lemma \ref{sym-dif}, every matching in $S$ must consist of a pair of opposite sides from every other square of the snake graph, starting with the first and ending with the last. Each of these pairs will be present in at least one matching of $S$.
		\end{lemma}
		\begin{proof}

			Let $\alpha$ be any square where two matchings $M_1,M_2\in S$ differ. By Lemma \ref{sym-dif}, every matching in $S$ must contain opposite sides of $\alpha$. Then $\alpha$ cannot be the second (or second to last) square of $G$, or else the edge $e\in\alpha$ with $\ell(e)$ non-boundary and unique in $G$ would have $\eta_{\ell(e)}$ between $0$ and $1$. 
			
			Next, assume $\alpha$ is not one of the first two or last squares of $G$, and let edges $e_1,e_2\in\alpha$ have $\ell(e_1),\ell(e_2)$ non-boundary segments. Let $\alpha_1$ be the square two before $\alpha$ and let $\alpha_2$ be the square two after $\alpha$. Let edges $e_1'\in\alpha_1$ and $e_2'\in\alpha_2$ have $\ell(e_1)=\ell(e_1')$ and $\ell(e_2)=\ell(e_2')$. Without loss of generality, $e_1,e_2\in M_1$ and $e_1,e_2\notin M_2$, so we must have $\{(w_{\text{nf}}^{M})_{\ell(e_1)}\}_{M\in S}=\{(w_{\text{nf}}^{M})_{\ell(e_2)}\}_{M\in S}=\{0,1,2\}$. Thus, some elements of $S$ must differ on whether they include $e_1'$ and $e_2'$, so every matching in $S$ must contain opposite sides of $\alpha_1$ and $\alpha_2$. Continuing recursively yields the lemma.
			
		\end{proof}
		Now we can show that all non-vertex lattice points are weight vectors of matchings.
		\begin{theorem}\label{NFASat} Let $\mca^{\text{nf}}$ be a cluster algebra of type $A$ with no frozen variables. Then the Newton polytope of any cluster variable, written as a Laurent polynomial in an arbitrary seed, is saturated.
		\end{theorem}
		\begin{proof}

			Let $T$ be a triangulation of $\apoly$ and $\gamma$ an arc. Again recall that by Lemma~\ref{lem:saturationEquiv}, $N^{\text{nf}}(T,\gamma)$ is saturated if and only if $P^{\text{nf}}(G_{T,\gamma})$ is saturated. Let $E$ be the set of edges $e$ appearing in the alternating squares described in Lemma \ref{alternating} such that $\ell(e)$ is not a boundary segment. By that lemma, the only possible non-vertex lattice point of $P^{\text{nf}}(G)$ is $\eta$ with $\eta_{\tau}=1$ if there exists $e\in E$ with $\ell(e)=\tau$ and $\eta_{\tau}=0$ otherwise. We can always construct $M$ with $w_{\text{nf}}^M=\eta$ by picking a pair of opposite sides from each of the alternating squares.
			
		\end{proof}
		Emptiness no longer holds in general, but Lemma \ref{alternating} tells us exactly what non-vertex lattice points must look like.
		\begin{cor}\label{NFAEmpty}
			Let $\mca^{\text{nf}}$ be a cluster algebra of type $A$ with no frozen variables. Let $T$ be a triangulation of $\apoly$ and $\gamma$ an arc. Then the Newton polytope of the cluster variable $L^{\text{nf}}_{T, \gamma}$ is empty if and only if at least one of the following conditions does not hold.
			\begin{itemize}
				\item $\gamma$ intersects all arcs in $T$.
				\item $n$ is odd and at least $3$.
				\item If $\gamma$ intersects $\tau_1,...,\tau_n$ in that order, $\tau_{i-1}$ and $\tau_{i+1}$ do not share an endpoint for even $i$.
			\end{itemize}
		\end{cor}
		\begin{proof}
			%
			%
			%
			%
			
			Each of these conditions is necessary for the existence of $M_1,M_2\in S$ as described in Lemma \ref{alternating}: the first ensures that interior edges of $G$ are labeled by non-boundary segments and the latter two ensure that the alternate squares do not share a corner. And by the final construction in the proof of Lemma \ref{alternating}, if all the stated conditions hold we can construct matchings $M_1$ and $M_2$ with $w_{\text{nf}}^{M_1}\neq w_{\text{nf}}^{M_2}$ but $\frac{1}{2}w_{\text{nf}}^{M_1}+\frac{1}{2}w_{\text{nf}}^{M_2}\in\mathbb{Z}^{n}$.
		\end{proof}
		\begin{remark}
			We observe that the results for type $A$ cluster algebras with no frozen variables are considerably less elegant than for our other two choices of frozen variables. Therefore, we focus on the latter two cases for type $D$ cluster algebras.
		\end{remark}
	\section{Type D cluster variable Newton polytopes}
		Now we turn to cluster variables in cluster algebras of type $D$; that is, to tagged arcs $\gamma$ and triangulations $T$ of $\dpoly$. Recall that we assume $T$ corresponds to an ideal triangulation (see Remark~\ref{rmk:idealOnly}). 
		\subsection{Boundary frozen variables}
		We will show that $N^{\text{bd}}(T,\gamma)$ is always saturated (Theorem \ref{DB-always}) considering the case of plain arcs (Proposition \ref{DB-saturated}) and notched arcs (Lemma \ref{DB-notched}) separately. In type $D$, Lemma \ref{map-pi} no longer holds. As a result, our proofs in this section are somewhat more involved. 
		
		For this section, let $\pi=\pi^{\text{bd}}$, $P(G)=P^{\text{bd}}(G)$, and $w^M=w_{\text{bd}}^M$.
		
		Our proof strategy is strong induction. We use the following observation in the inductive step. Recall the definition of a non-corner square of a snake graph (Definition~\ref{def:corner}).
		
		\begin{observation}\label{reduction}
			Let $G$ be a snake graph and let $H$ be an ``initial" or ``final" sub-snake graph of $G$ which is connected to the rest of $G$ by a non-corner square with outer edges $e_1$ and $e_2$. The graph $K:= G \setminus (H \cup \{e_1, e_2\})$ is also a snake graph, and the union of any matchings of $H$ and $K$ is a matching of $G$. If $P(K)$ is saturated, then to show a lattice point $v \in P(G)$ is the weight vector of a matching, it suffices to find $\overline{v}\in\pi^{-1}(v)$ where $\overline{v}|_{E(H)}$ gives a matching of $H$.
		\end{observation}
		
		We begin with some lemmas giving sufficient conditions for saturation of $P(G)$ in terms of saturation of $P(G')$, for $G'$ a sub-snake graph.
		
		\begin{figure}
			\centering
			\includegraphics[width=0.7\linewidth]{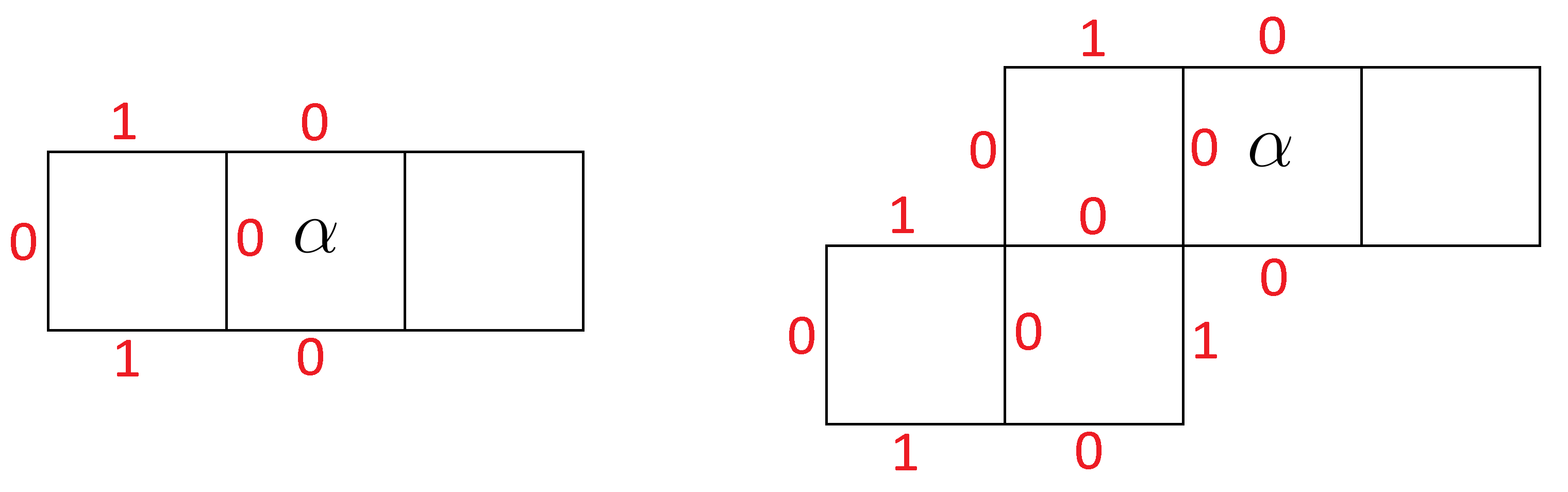}
			\caption{Two examples of the case of $\overline{v}_{e_2}=\overline{v}_{e_4}=1$ in Lemma \ref{chip1}. The numbers in red are coordinates of $\overline{v}$. Note that up to the second non-corner square, all coordinates are determined by the fact that the sum around each vertex must be $1$. }
			\label{fig:horizontal-chip}
		\end{figure}

		\begin{lemma}\label{chip1}
			For a snake graph $G$, suppose that the first tile of $G$ has edges $e_1,e_2,e_3,e_4$ (clockwise) such that $e_1$ connects the first tile to another tile in $G$ and the label of any one of $e_2,e_3,e_4$ is unique in the diagram. Let $G'$ be a snake graph obtained by deleting some boxes from that end of $G$, inheriting the edge labels of $G$. Then $P(G)$ will be saturated if $P(G')$ is saturated for all such $G'$.
		\end{lemma}
		\begin{proof}	

			Assume saturation of all such $P(G')$, let $v\in P(G)$ be a lattice point, and let $\overline{v}\in \overline{P}(G)$ lie in $\pi^{-1}(v)$. By the given uniqueness of a label, at least one of $\overline{v}_{e_2},\overline{v}_{e_3},\overline{v}_{e_4}$ is $0$ or $1$, and so either: $\overline{v}_{e_2}=\overline{v}_{e_4}=0$ and $\overline{v}_{e_3}=1$, or $\overline{v}_{e_2}=\overline{v}_{e_4}=1$ and $\overline{v}_{e_3}=0$. 
			
			If $\overline{v}_{e_2}=\overline{v}_{e_4}=0$, we are in the scenario of Observation \ref{reduction}, where $H$ is $e_3$ and $K$ is the graph obtained from $G$ by deleting $e_2, e_3, e_4$ ($P(K)$ is saturated by assumption). So using Observation \ref{reduction}, we can find a matching $M$ of $G$ with weight vector $v$.
			
			Next, consider the latter case of $\overline{v}_{e_2}=\overline{v}_{e_4}=1$. Let $H$ be the sub-snake graph of $G$ which is the union of tiles from the first tile up to (but not including) the next non-corner tile, which we call $\alpha$. Each of the coordinates of $\overline{v}$ corresponding to edges of $H$ must be $0$ or $1$ (see Figure \ref{fig:horizontal-chip}). Further, the coordinates indexed by the border edges of $\alpha$ must be $0$. Again, we are in the scenario of Observation~\ref{reduction}, and so can find a matching with weight vector $v$.
		\end{proof}
		
		\begin{figure}
			\centering
			\includegraphics[width=0.7\linewidth]{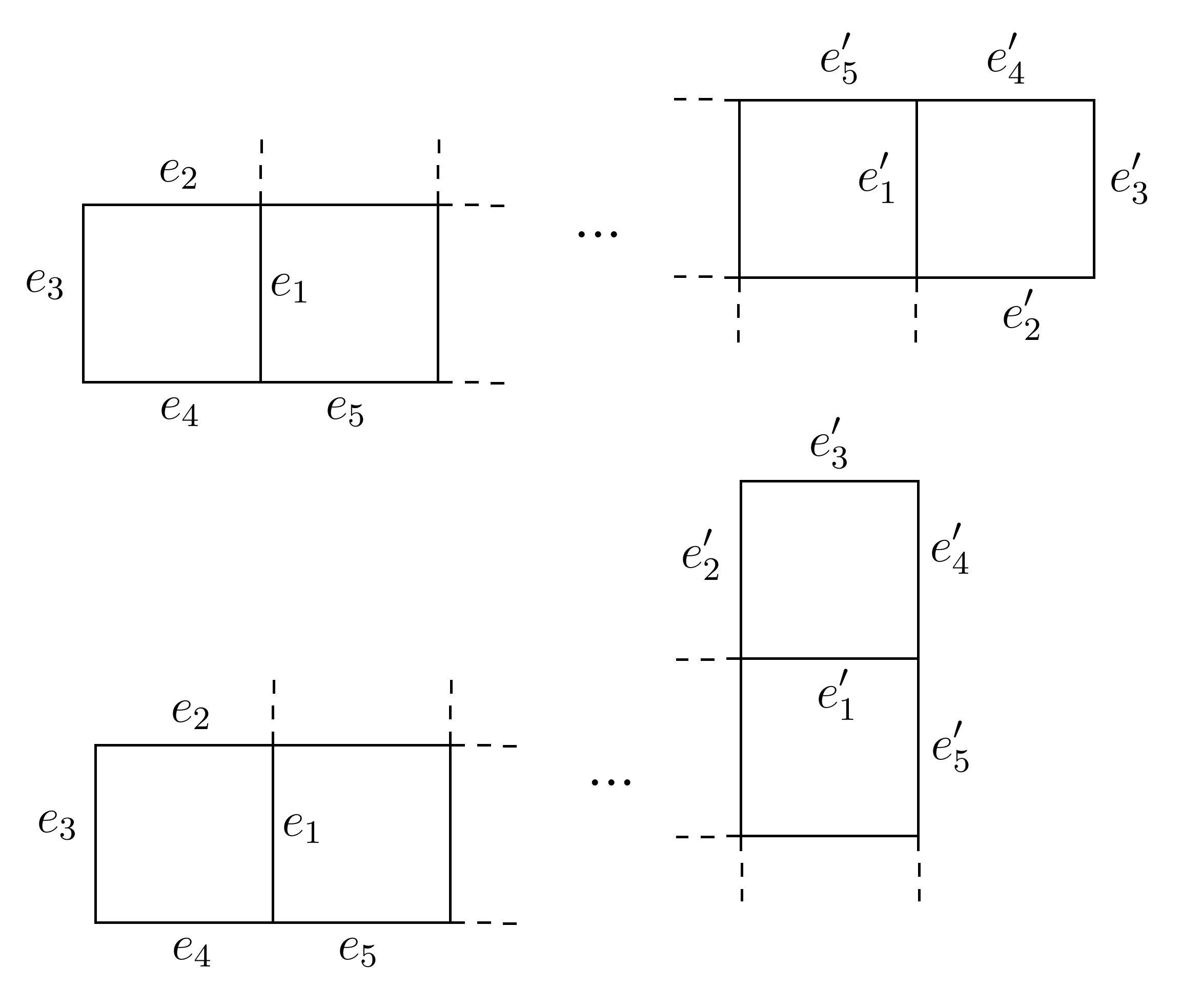}
			\caption{Ends of a snake graph satisfying the symmetry condition in the hypothesis of Lemma \ref{chip2}.}
			\label{fig:symmetric-ends}
		\end{figure}

		\begin{lemma}\label{chip2}
			For a snake graph $G$, suppose the first two boxes at either end of $G$ are as shown in Figure \ref{fig:symmetric-ends}, where $\ell(e_i)=\ell(e'_i)$. Suppose further that $\ell (e_5) $ and one of $\ell(e_2), \ell(e_3), \ell(e_4)$ occur exactly twice in the diagram. Let $G'$ be a snake graph obtained by deleting one or more boxes from either end of $G$. Then $P(G)$ is saturated if $P(G')$ is saturated for all such $G'$.
		\end{lemma}
		
		\begin{proof}

			Assume saturation of all such $P(G')$, let $v\in P(G)$ be a lattice point, and let $\overline{v}\in \overline{P}(G)$ lie in $\pi^{-1}(v)$. Let $e_{i}\in\{e_2,e_3,e_4\}$ have the edge label that occurs exactly twice, so $\overline{v}_{e_i}+\overline{v}_{e'_i}\in\{0,1,2\}$. If this sum is $0$ or $2$ then $\overline{v}_{e_i}$ is $0$ or $1$ respectively, and the remainder of the proof is exactly like the proof of Lemma~\ref{chip1}. Otherwise,  $\overline{v}_{e_j}=1-\overline{v}_{e'_j}$ for $j\in\{2,3,4\}$. Let $x:=\overline{v}_{e_2}$ and note $x=\overline{v}_{e_4}=1-\overline{v}_{e_3}$.  
			
			Since the label of $e_{5}$ occurs exactly twice, we similarly have $\overline{v}_{e_5}+\overline{v}_{e'_5}\in \{0,1,2\}$. If it is $2$, then $\overline{v}_{e_5}=1$, so $x=0$, and we can again proceed by Observation \ref{reduction}. If it $\overline{v}_{e_5}+\overline{v}_{e'_5}=0$, then $\overline{v}_{e_5}=\overline{v}_{e'_5}=0$, so $\overline{v}_{e_1}=1-x=1-\overline{v}_{e'_1}$. Define $\overline{q}\in \pi^{-1}(v)$ by $\overline{q}_{e_1}=\overline{q}_{e_3}=\overline{q}_{e'_2}=\overline{q}_{e'_4}=1$, $\overline{q}_{e'_1}=\overline{q}_{e'_3}=\overline{q}_{e_2}=\overline{q}_{e_4}=0$, and $\overline{q}_{e_i}=\overline{v}_{e_i}$ for all other coordinates. Since $\overline{q}_{e_2}=\overline{q}_{e_4}=0$ we can proceed by Observation \ref{reduction}, where $H$ is $e_3$.
			
			Finally, suppose $\overline{v}_{e_5}+\overline{v}_{e'_5}=1$ and let $\overline{v}_{e_5}=y$. We see that $x+y=1$ so $\overline{v}_{e_1}=\overline{v}_{e'_1}=0$ and $\overline{v}_{e_5}=y=1-x$. Now, let $G'$ be $G$ with the first and last box removed. Define $\overline{v}'\in \overline{P}(G')$ by $\overline{v}'_{e_1}=x$, $\overline{v}'_{e'_1}=1-x$, and $\overline{v}'_{e_i}=\overline{v}_{e_i}$ for all other coordinates. Since $P(G')$ is saturated by assumption, there is some lattice point $\overline{q}'\in \overline{P}(G')$ such that $\pi(\overline{v}')=\pi(\overline{q}')$. Without loss of generality let $\overline{q}'_{e_5}=0$ and $\overline{q}'_{e'_5}=1$, so $\overline{q}'_{e_1}=1$ and $\overline{q}'_{e'_1}=0$. Define the lattice point $\overline{q}\in \overline{P}(G)$ by $\overline{q}_{e_1}=\overline{q}_{e_3}=\overline{q}_{e'_2}=\overline{q}_{e'_4}=0$ and $\overline{q}_{e'_3}=\overline{q}_{e_2}=\overline{q}_{e_4}=1$ and $\overline{q}_{e_i}=\overline{q}'_{e_i}$ for all other coordinates. Then, $\pi(\overline{q})=\pi(\overline{v})=v$, proving saturation of $G$. 
		\end{proof}
		
		Now, in order to make use of the previous lemmas, we must verify which snake graphs $G_{T, \gamma}$ satisfy the hypotheses. First, we analyze the structure of ideal triangulations of $\dpoly$. 
		
		Fixing a triangulation $T$, let $\rho_1,...,\rho_k$ be the radii, ordered counterclockwise, and let $b_i$ be the basepoint of $\rho_i$. Note $T$ contains an edge $\varepsilon_i$ between each $b_i$ and $b_{i+1}$, or else the space between $\rho_i$ and $\rho_{i+1}$ would not be triangulated. Note that $\varepsilon_1,...,\varepsilon_k$ form the boundary of a punctured $k$-gon triangulated by radii. Let $B_i$ be the set of boundary vertices between $b_i$ and $b_{i+1}$ inclusive. Since no edge of $T$ can cross one of the radii, two boundary vertices can be connected by an edge of $T$ only if they are both in $B_i$. For $k>1$, let $S_i$ be the surface enclosed by $\varepsilon_i$ and boundary edges between $b_i$ and $b_{i+1}$. Note that $S_i$ is isotopic to an un-punctured convex polygon and $T|_{S_i}$ is a triangulation. Thus, for $k>1$, $T$ is combinatorially equivalent to a punctured $k$-gon triangulated by $k$ radii, with a triangulated polygon glued to each boundary segment. See Figure \ref{fig:triangulation} for examples.
		\begin{figure}
			\centering
			\includegraphics[width=0.7\linewidth]{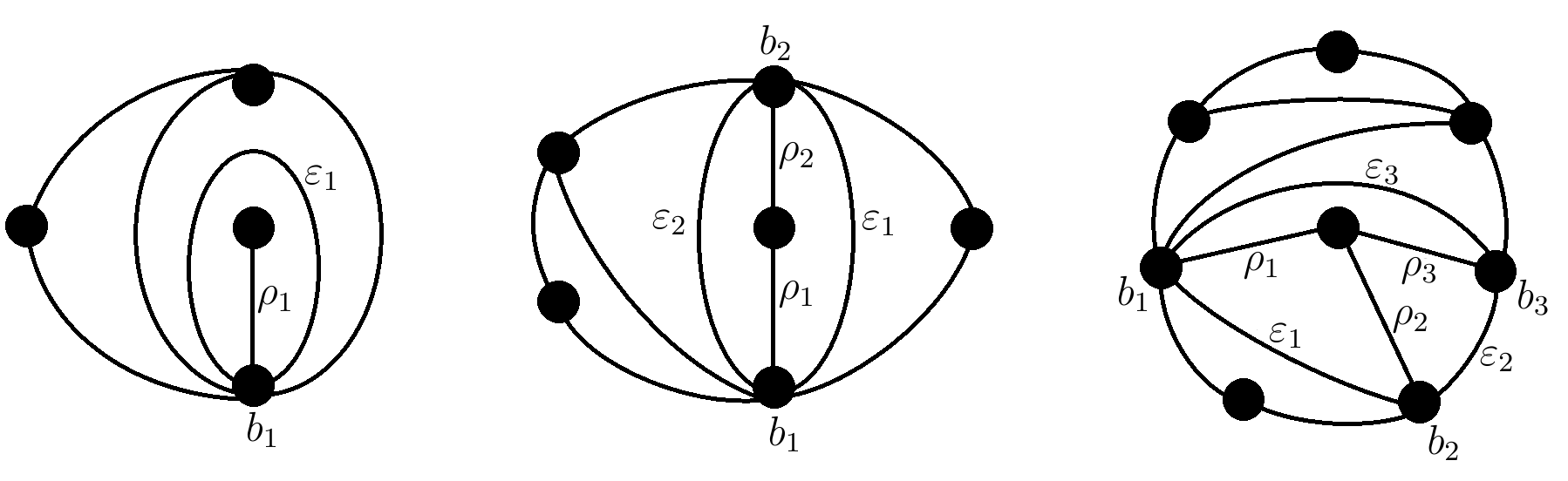}
			\caption{Examples of our notation for triangulations of punctured polygons with $k=1$, $k=2$, and $k=3$ radii.}
			\label{fig:triangulation}
		\end{figure}
		
		Let $\gamma$ be an oriented arc not in $T$. We will describe its intersections with the edges of $T$. If $\gamma$ is a radius oriented from base point $V$ to the puncture, it will only intersect edges in some $S_i$ followed by $\varepsilon_i$; denote the ordered set of these edges by $I_V$. If the endpoints of $\gamma$ are boundary vertices $W,W'$ with $W\in B_i,W'\in B_j$ (possibly the same), then either $\gamma$ will be a diagonal of a polygon, or it will intersect $T$ at $I_{W}$, followed by the radii $\rho_{i}$ through $\rho_{j-1}$,\footnote{Note that this is slightly imprecise if $W$ or $W'$ is in $\{b_i\}_{i=1}^{k}$ since $i$ and $j$ can be either of two values.} followed by $I_{W'}$ in reverse order. Which of these two cases applies depends on which side of the puncture $\gamma$ lies.
		
		If $k>1$ and $\gamma$ intersects some $\varepsilon_i$ twice, the first and last radius in $T$ that $\gamma$ intersects will both correspond to non-corner squares of $G_{T, \gamma}$. This is true since $\gamma$ will consecutively intersect $\varepsilon_i,\rho_i,\rho_{i-1}$ which do not share a common endpoint, and similarly on the other side. Furthermore, all other radii in $T$ which $\gamma$ intersects will correspond to corner squares, since those arcs share a common endpoint at the puncture.  
		
		Next, we use our analysis above to develop some results about edges with unique (or almost unique) labels.
		\begin{prop}\label{unique}
			Let $T$ be an ideal triangulation of $\dpoly$, and let $\gamma$ be an arc that is either not a loop or is a loop enclosing a radius in $T$. Then there is at least one edge on the boundary of the initial or final box of $G_{T,\gamma}$ whose label is unique in $G_{T,\gamma}$.
		\end{prop}
		\begin{proof}
			Let the end points of $\gamma$ be $W$ and $W'$, and let $\rho_{V}$ be the radius based at $V$ (not necessarily in $T$). We prove each of five cases separately. 
			\begin{enumerate}
				\item There is no $i$ with $W,W'\in B_i$, and at least one of $W,W'\notin \{b_i\}_{i=1}^k$. 
				
				Assuming $W$ is a boundary vertex, each of the two arcs incident to $W$ surrounding $\gamma$ will label unique edges of $G_{T, \gamma}$ due to the type $A$ structure near the end points. Note that this includes the case where $W'$ is the puncture.
				\item Both $W,W'\in \{b_i\}_{i=1}^k$. 
				
				The subset of $T$ that contributes to $G_{T,\gamma}$ is isotopic to a (un-punctured) triangulated polygon, so the two arcs incident to $W$ (resp. $W'$) and surrounding $\gamma$ will label unique edges.
				
				\item $W,W'\in B_i$ and $\gamma$ does not intersect $\varepsilon_i$.
				
				In this case $\gamma$ is contained in $S_i$, which is equivalent to an arc in an un-punctured polygon, so again the arcs incident to $W, W'$ and surrounding $\gamma$ will label unique edges.
				
				\item $W,W'\in B_i$ and $\gamma$ intersects $\varepsilon_i$. 
				
				Without loss of generality, $W\notin \{b_i\}_{i=1}^k$ so the arcs surrounding $\gamma$ incident to $W$ are between elements of $B_i$. These will not be edge labels of tiles corresponding to radii, so it suffices to consider tiles corresponding to edges in $I_W$ and $I_{W'}$. Let $T'\subset T$ consist of arcs labeling edges of $G_{T,\rho_{W}}$. If $\rho_{W}'$ is contained inside the bounds of $T'$, then the arcs surrounding $\gamma$ containing $W$ are unique labels. Otherwise, the arcs surrounding $\gamma$ containing $W'$ are unique labels. 
				\item $W=W'$
				
				We have $W=W'=b_i$, since $\gamma$ encloses a radius of $T$. Thus, $\gamma$ is contained in the $k$-gon with sides $\{\varepsilon_i\}$ and intersects all radii but $\rho_i$. Both $\varepsilon_i$ and $\varepsilon_{i-1}$ label a boundary edge in an end square of $G_{T,\gamma}$, and these labels will be unique in the diagram.  
			\end{enumerate}
		\end{proof}
		
		Now we address two pathological cases.
		\begin{lemma}\label{pseudo-loop}
			Let $T$ be an ideal triangulation of $\dpoly$, and let $\gamma$ be a loop that intersects exactly one non-radius arc of $T$. Let $G'$ be the snake graph obtained by removing the first and last square from $G_{T,\gamma}$ (as shown in Figure \ref{fig:staircase} when $k=6$). Then $P(G')$ is saturated.
		\end{lemma}
		\begin{proof}
			%
			%
			
			If $k=1$, then $G'$ consists of a single box with two adjacent pairs of edges labeled $\rho_1$ and $\varepsilon_1$ respectively; $P(G')$ is a single point and so is saturated.
			
			For $k\geq2$, the tiles of $G'$ alternate between horizontal and vertical adjacency, as shown in Figure \ref{fig:staircase}.\footnote{For $k=2$ some of the following structure degenerates, but the proof still holds.} The internal edges of $G'$ are labeled by $\varepsilon_2,...,\varepsilon_k$ and these labels are unique. All other labels appear exactly twice. Whenever two boundary edges meet at a vertex that is also incident to two internal edges, those two boundary edges are both labeled by $\rho_i$ for some $i=3,...k$. 
			\begin{figure}
				\begin{minipage}{0.6\linewidth}
					\centering
					\includegraphics[width=0.8\linewidth]{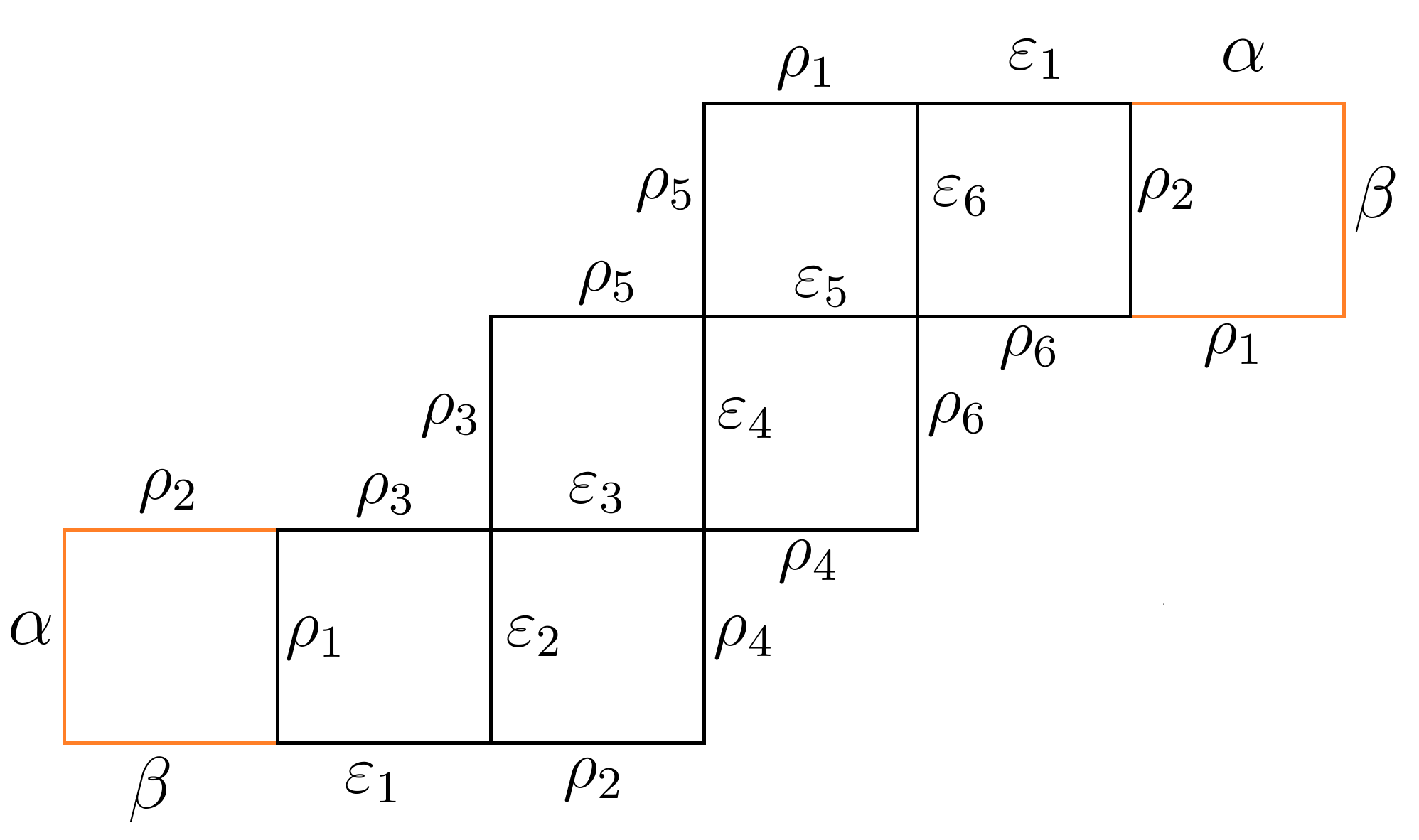}
				\end{minipage}
				\begin{minipage}{0.3\linewidth}
					\centering
					\includegraphics[width=0.7\linewidth]{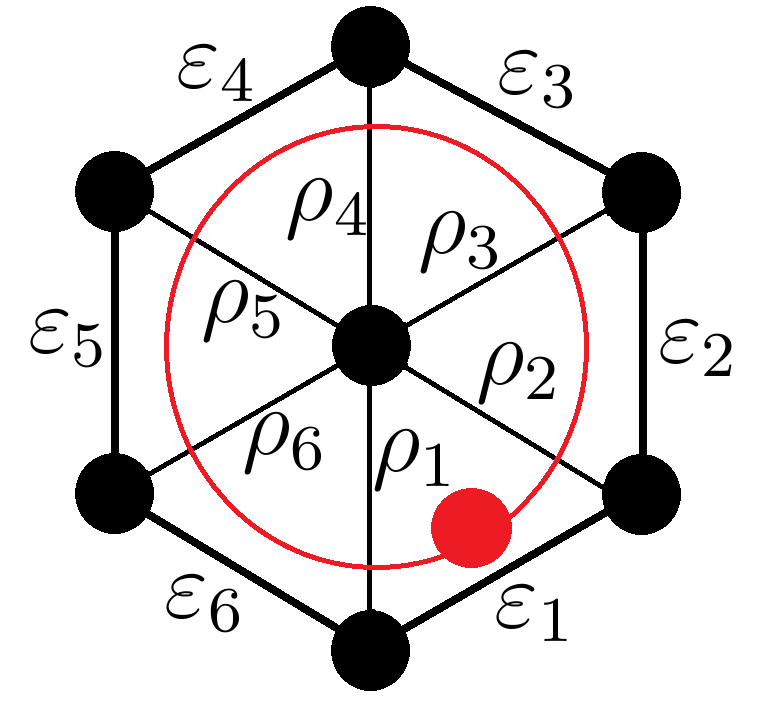}
				\end{minipage}
				\caption{On the left is the middle part of the snake graph of a loop (with base vertex not in $\{b_i\}_{i=1}^k$) in the case $k=6$. The tiles correspond to intersections with all radii; tiles corresponding to non-radii have been removed. On the right is a ``floating arc'' whose intersections with the the radii would correspond to this snake graph.}
				\label{fig:staircase}
			\end{figure}	
			
			Let $\overline{v}\in \overline{P}(G')$ such that $v=\pi(\overline{v})$ is a lattice point in $P(G')$. First, consider the case $\overline{v}_{\varepsilon_i}=0$ for all $i=2,...,k$. Then any pair of coordinates of $\overline{v}$ corresponding to edges with the same label must sum to $1$. Thus, $v_{\tau}=1$ for all $\tau$ appearing more than once in $G'$, and $v$ is the weight vector of either of the two perfect matchings that use only boundary edges of $G'$.  
			
			On the other hand, consider the case $\overline{v}_{\varepsilon_i}=1$ for some $i\in\{2,...,k\}$. Then $\overline{v}_{\tau}=0$ for all edges labeled $\tau$ incident to $\varepsilon_i$. By the equations of $\overline{P}(G')$ corresponding to degree $2$ vertices of $G'$, we see that all coordinates of $\overline{v}$ are either $0$ or $1$, meaning that $\overline{v}$ corresponds to a perfect matching.  
		\end{proof}
		\begin{lemma}\label{loop1}
			Let $T$ be an ideal triangulation of $\dpoly$. If $\gamma\notin T$ is a loop that intersects exactly one non-radius edge of $T$, then $P(G_{T,\gamma})$ is saturated.
		\end{lemma}
		\begin{proof}
			Without loss of generality, $\gamma$ intersects $\varepsilon_1$, then all of $\{\rho_i\}_i$, and finally $\varepsilon_1$ again.
			
			If $k=1$, then $\varepsilon_1$ is a loop and $G_{T,\gamma}$ is as shown in Figure \ref{fig:triple-tile}. We see saturation directly, as the only weight vectors (in the coordinates $(\alpha,\beta,\rho_1,\varepsilon_1)$) are $(2,0,1,1)$, $(1,1,1,1)$, $(0,2,1,1)$.  
			
			For $k\geq 2$, the first two and last two squares of $G_{T,\gamma}$ are as shown in Figure \ref{fig:loop-ends}, with edges differing by a prime mark labeled with the same arc. Arcs $\alpha$, $\beta$, and $\varepsilon_1$ do not appear elsewhere in the diagram. 
			\begin{figure}
				\centering
				\includegraphics[width=0.5\linewidth]{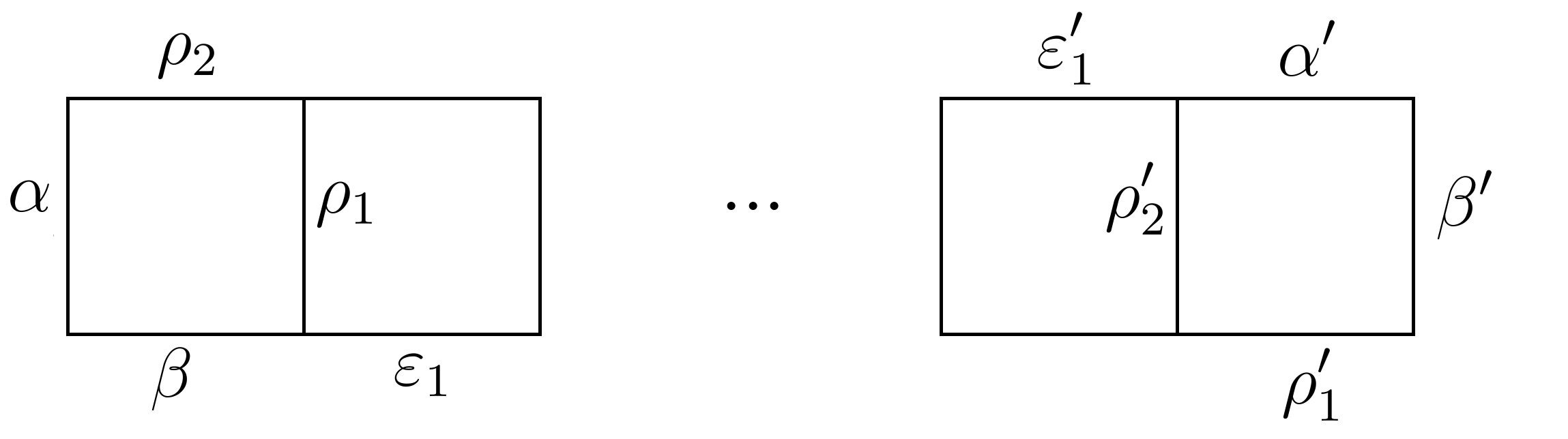}
				\caption{The ends of the snake graph of a loop intersecting only $\varepsilon_1$ and radii, in the case $k\geq 2$. }
				\label{fig:loop-ends}
			\end{figure}
			
			Let $v\in P(G_{T,\gamma})$ be a lattice point and $\overline{v}\in \pi^{-1}(v)$. We have $\overline{v}_{\alpha}+\overline{v}_{\alpha'}\in\{0,1,2\}$ and $\overline{v}_{\varepsilon_1}+\overline{v}_{\varepsilon'_1}\in\{0,1,2\}$, and if any of these four coordinates of $\overline{v}$ are integers we can apply the proof of Lemma \ref{chip2} directly after we show that $P(G')$ is saturated for all $G'\subset G$. If $G'$ is obtained by removing one or more squares from one end of $G_{T,\gamma}$, then $P(G')$ is saturated by Lemma \ref{chip1}. If $G'$ is obtained by removing one square from each end of $G_{T,\gamma}$, then $P(G')$ is saturated by Lemma  \ref{pseudo-loop}. And if $G'$ is obtained by removing any more squares from each end of $G_{T,\gamma}$, then $G'=G_{T,\gamma'}$ where $\gamma'$ is an arc of an un-punctured polygon, so $P(G')$ is saturated by Proposition \ref{AB-saturated}. 
			
			Now, suppose $x=1-\overline{v}_{\alpha}=\overline{v}_{\beta}=\overline{v}_{\rho_2}=\overline{v}_{\alpha'}=1-\overline{v}_{\beta'}=\overline{v}_{\rho'_1}$ and $y=\overline{v}_{\varepsilon_1}=1-\overline{v}_{\varepsilon'_1}$ with $0<x,y<1$. Since the labels $\varepsilon_2,...,\varepsilon_k$ that are unique in the diagram, the corresponding coordinates of $\overline{v}$ must be integers. As in the proof of Lemma \ref{pseudo-loop}, if $\overline{v}_{\varepsilon_i}=1$ for some $i$, then $\overline{v}_{\varepsilon_1}$ and $\overline{v}_{\varepsilon'_1}$ are also integers so $\overline{v}$ is a lattice point as desired. If instead $\overline{v}_{\varepsilon_i}=0$ for $i=2,...,k$, the equations defining $\overline{P}(G_{T,\gamma})$ imply that $v_{\alpha}=v_{\beta}=v_{\rho_i}=1$. And this lattice point $v$ corresponds to the matching that includes $\alpha$, $\beta'$, and the bottom matching of $G'$, where $G'$ is the snake graph obtained by removing the first and last squares of $G_{T,\gamma}$. 
		\end{proof}
		
		\begin{figure}[h]
			\centering
			\includegraphics[width=0.7\linewidth]{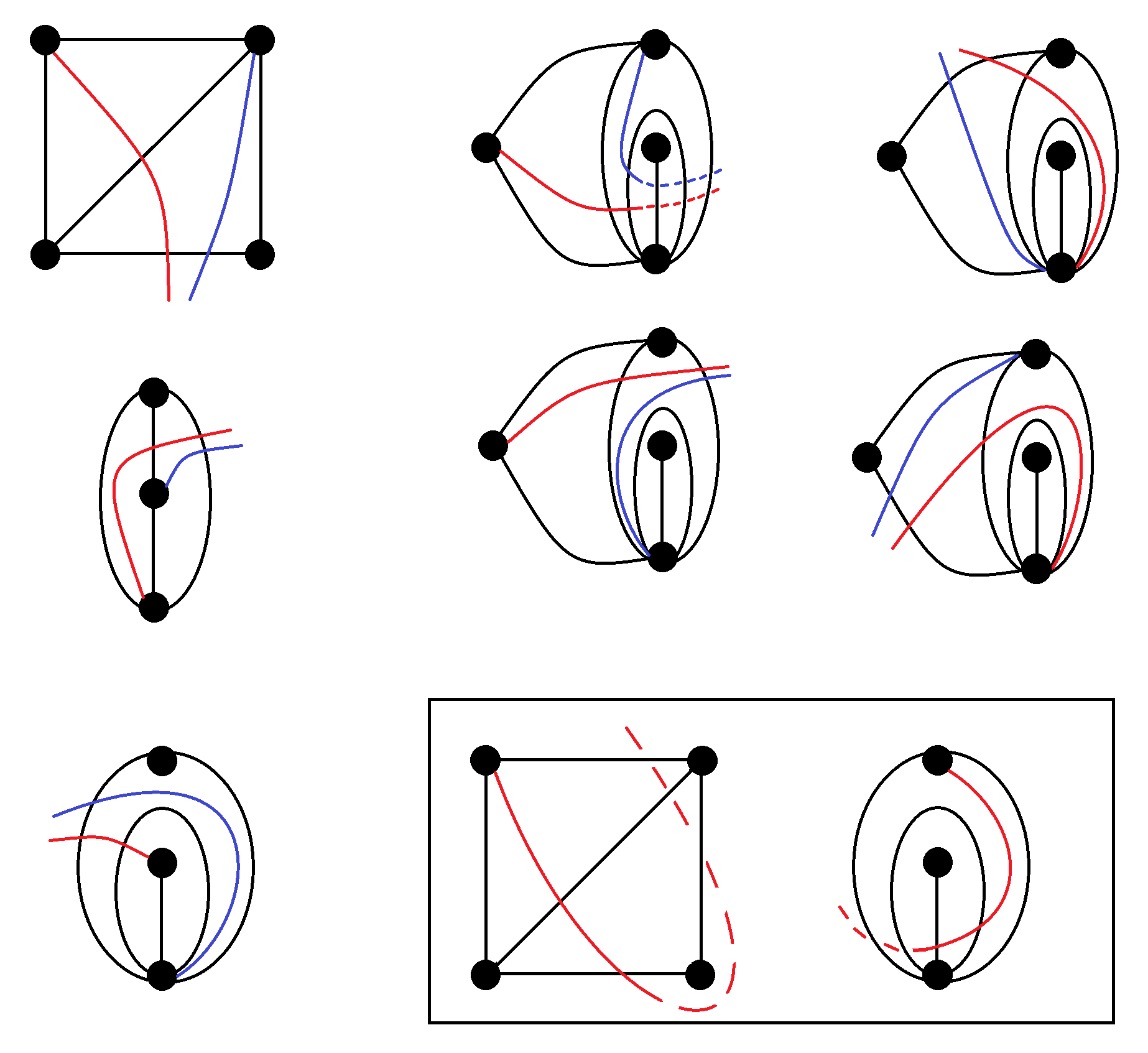}
			\caption{The red curves show all possible ends of arcs $\gamma$, up to symmetry, that intersect at least two edges of an ideal triangulation $T$. In each case, the blue curve is the end of an arc $\gamma'$ such that $P(G_{T,\gamma'})$ is $P(G_{T,\gamma})$ with the last square removed from the given end. Such an arc $\gamma'$ exists in all cases except one (in general as shown in the box on the left, with the case for $k=1$ on the right).}
			\label{fig:shifts}
		\end{figure}
		
		The next lemma will be the key ingredient of our strong induction. 
		\begin{lemma}\label{strong-induction}
			Let $T$ be an ideal triangulation of a punctured polygon and $\gamma$ an arc. Let $G'$ be any snake graph obtained by deleting at least one square from one or both ends of $G_{T,\gamma}$. Then $G'=G_{T,\gamma'}$ for some arc $\gamma'$, or $P(G')$ is saturated.
		\end{lemma}
		\begin{proof}

			Figure \ref{fig:shifts} shows how given an end of an arc $\gamma$ we can find another arc $\gamma'$ such that $G_{T,\gamma'}$ the graph obtained by removing one square from the appropriate end of $G_{T,\gamma}$ --- in all cases but those in the box. Those cases occur if and only if the diagonal of the quadrilateral is $\varepsilon_i$ and two of the sides are radii (the radii coincide for the case $k=1$ shown on the right), because $\gamma'$ would have to intersect itself and so wouldn't be a valid arc. For the boxed case on the left, shifting the endpoint of $\gamma$ to the bottom left vertex of the quadrilateral corresponds to removing two squares from one end of $G_{T,\gamma}$ (and likewise for the $k=1$ case by shifting the endpoint of $\gamma$ to the puncture). It remains to show the following:
			\begin{itemize}
				\item If $\gamma$ is a loop with both ends as in the boxed case, removing a square from each end of $G_{T,\gamma}$ yields $G'$ with $P(G')$ saturated.
				\item If $\gamma$ has at least one end as in the boxed case, removing one square of $G_{T,\gamma}$ from that end yields $G''$ with $P(G'')$ saturated.
			\end{itemize}
			In the first case, $G'$ consists only of squares corresponding to radii, and Lemma \ref{pseudo-loop} showed that $P(G')$ is saturated. 
			
			In the second case, the end of $G''$ contains $G_{T,\rho}$ for some radius $\rho$ with base point at the endpoint of $\gamma$. By Proposition~\ref{unique}, at least one boundary edge of the last square of $G''$ is unique in the diagram, so we can use Observation \ref{reduction} to reduce saturation of $G''$ to saturation of a smaller diagram, iterating until the diagram is either $G'$ from the first case or $G_{T,\gamma'}$ for a non-loop arc $\gamma'$ intersecting only radii. In the latter case, $\gamma'$ is an arc of an un-punctured polygon, so $P(G_{T,\gamma'})$ is saturated by Proposition \ref{AB-saturated}.
		\end{proof}
		
		Now, we have all of the base cases and reduction techniques for a proof by induction.
		\begin{proposition}\label{DB-saturated}
			Let $T$ be an ideal triangulation of $\dpoly$, and let $\gamma$ be an ordinary arc. Let $T'$ be the corresponding tagged triangulation. Then $N(T',\gamma)$ is saturated.
		\end{proposition}
		
		\begin{proof}

			By Lemma~\ref{lem:saturationEquiv}, it suffices to show that $P(T, \gamma)$ is saturated. We proceed by induction on the number of boxes in the snake diagram $G_{T,\gamma}$. For the base case where $\gamma$ intersects a single arc of $T$, $G_{T,\gamma}$ is a single tile and has an edge with a unique label, so $P(G_{T,\gamma})$ is saturated by Observation \ref{reduction}. 
			
			For the inductive step, assume that $P(G_{T,\gamma})$ is saturated for all $G_{T,\gamma}$ with at most $t$ boxes. Suppose there is an arc $\tau$ such that $G_{T,\tau}$ has $t+1$ boxes. If $\tau$ is not a loop, or intersects each arc of $T$ at most once, then Proposition~\ref{unique}, Lemma \ref{strong-induction}, and the inductive hypothesis ensure that the hypothesis of Lemma \ref{chip1} is satisfied, so $P(G_{T,\tau})$ is saturated. If $\tau$ is a loop that intersects exactly one non-radius arc, then $P(G_{T,\tau})$ is saturated by Lemma \ref{loop1}. If $\tau$ intersects at least two non-radius arcs, then by Proposition~\ref{unique}, Lemma \ref{strong-induction}, and the inductive hypothesis we can apply Lemma \ref{chip2} to conclude that $P(G_{T,\tau})$ is saturated.
		\end{proof}
		
		\begin{figure}
			\centering
			\includegraphics[width=0.5\linewidth]{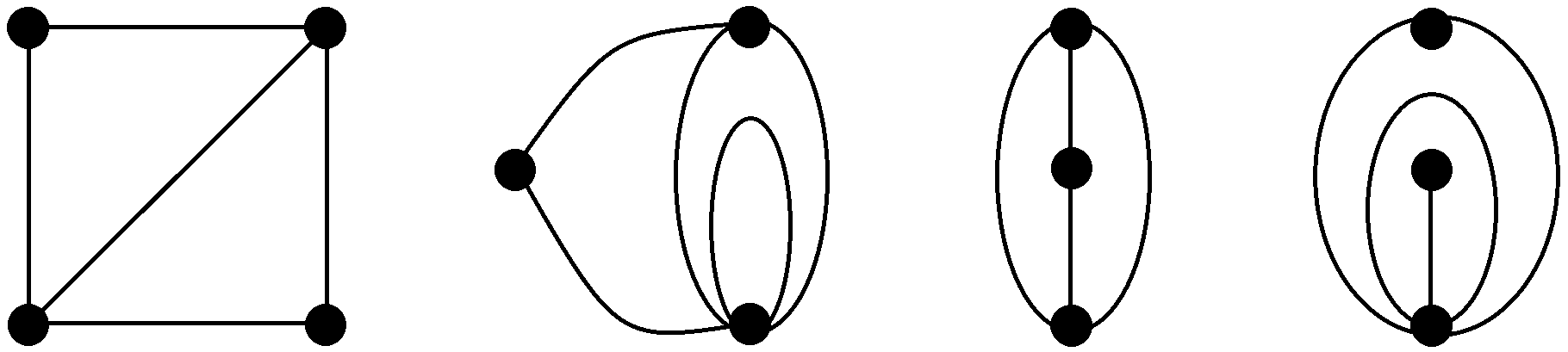}
			\caption{Types of quadrilaterals appearing in ideal triangulations of punctured polygons.}
			\label{fig:tiles}
		\end{figure}
		We turn to notched arcs, recalling the corresponding expansions of cluster variables are given in terms of $\rho$-symmetric matchings of snake graphs (see Section~\ref{subscn:notchedArcExpansion}). 
		
		\begin{proposition}\label{DB-notched}
			Let $T$ be an ideal triangulation of $\dpoly$, and let $T'$ be the corresponding tagged triangulation. Let $\rho^{\bowtie}$ be a notched arc. Then $N(T',\rho^{\bowtie})$ is saturated. 
		\end{proposition}
		\begin{proof}

			Let $\lambda$ be the loop enclosing the unnotched radius $\rho$. If the unnotched radius $\rho$ is in $T$, then $N(T', \rho^{\bowtie})$ is an integer translate of $N(T, \lambda)$ (see Remark~\ref{rmk:easyNotchedExpansions}), which is saturated by Proposition~\ref{DB-saturated}. If $T \neq T'$ (that is, if $T$ contains a loop and $T'=(T')^p$), then $N(T', \rho^{\bowtie})$ is, up to renaming coordinates, equal to $N(T', \rho)$ (see Remark~\ref{rmk:easyNotchedExpansions}). So we may assume $T=T'$ and $\rho \notin T$. 
			
			Let $\{M_i\}_i$ be a set of $\rho$-symmetric matchings of $G_{T,\lambda}$ such that the lattice point $v=\sum_{i}c_iw^{M_i}$ is a convex combination of weight vectors. By Proposition~\ref{unique} there is an edge label $\tau$ on the boundary in the first and last square of $G_{T,\lambda}$ that occurs exactly twice in $G_{T,\lambda}$. Any $\rho$-symmetric matching must either include or not include both edges labeled $\tau$, so $(w^{M_i})_{\tau}$ is $0$ or $1$ for all $M_i$ (recall the definition of $\owt(M)$). Since $v$ is an integer, either all $M_i$ include both edges labeled $\tau$ or none of them do.
			
			Arguing as in the proof of Lemma~\ref{chip1}, the edges in $M_i$ are determined from the first square to the next non-corner square on both ends of $G_{T, \lambda}$. Removing those squares from $G_{T,\lambda}$ yields a strictly smaller subgraph $G'$ with $\rho$-symmetric matchings $\{M_i'\}_i$. Note that adding the same set of edges to each $M_i'$ yields $\{M_i\}_i$, and $\sum_{i}c_iw^{M_i'}$ is a convex combination of weight vectors of $G'$ equal to a lattice point. By Lemma \ref{strong-induction}, we can repeat this argument to remove tiles from both ends of $G_{T,\lambda}$ until we have $G'$ containing no tiles from $G_{T,\rho,1}$ and $G_{T,\rho,2}$. Then the $\rho$-symmetry condition on $G'$ is vacuous, so $P(G')$ is saturated by Lemma \ref{strong-induction} and Proposition \ref{DB-saturated}. Thus, $\sum_{i}c_iw^{M_i'}=w^{M_j'}$ for one of the matchings $M_j$. And since $\{M_i'\}_i$ determines $\{M_i\}_i$ at each stage, we have $\sum_{i}c_iw^{M_i}=w^{M_j}$ as desired. 
		\end{proof}
		\begin{figure}
			\centering
			\includegraphics[width=0.2\linewidth]{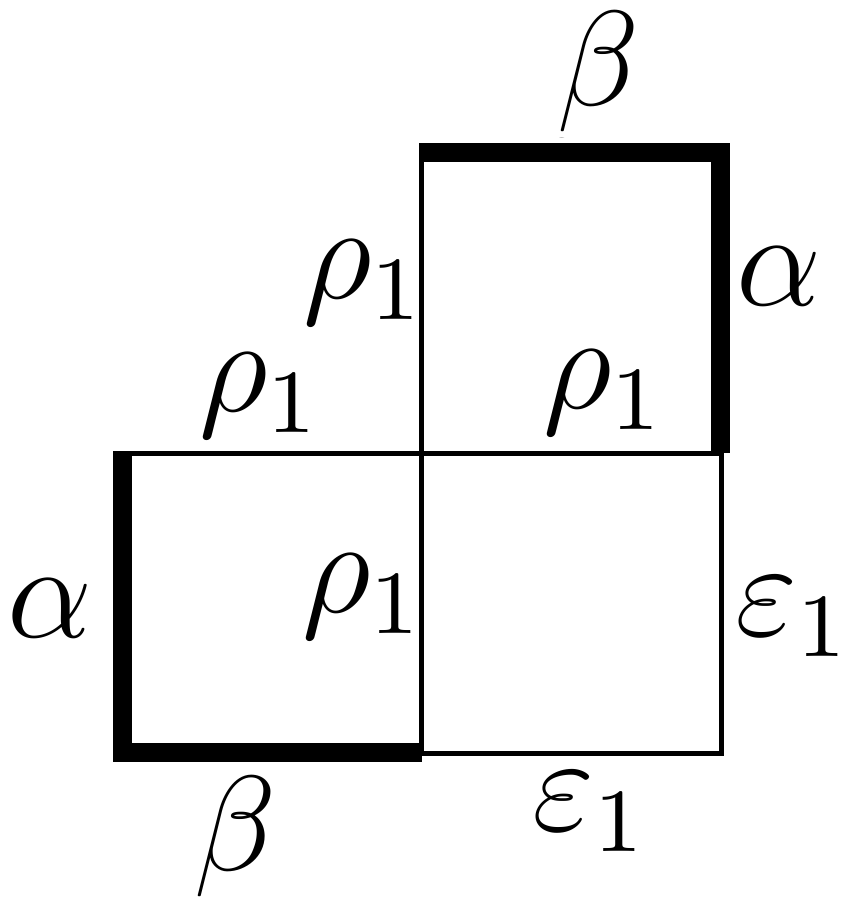}
			\caption{Middle part of $G_{T,\gamma}$ where $\gamma$ is a loop (with base point not in $\{b_i\}$) and $k=1$, with edges of $G_1$ and $G_2$ in bold.}
			\label{fig:triple-tile}
		\end{figure}
		Proposition \ref{DB-saturated} and Proposition \ref{DB-notched}, together with Remarks~\ref{rmk:idealOnly}, give the following result.
		\begin{theorem}\label{DB-always}
			Let $\mca^{\text{bd}}$ be a type $D$ cluster algebra with boundary coefficients. Then the Newton polytope of any cluster variable, written as a Laurent polynomial in an arbitrary seed, is saturated and empty.
		\end{theorem}
		
		Next, we classify when $N^{\text{bd}}(T,\gamma)$ is empty. The following simple lemma will be very handy.
		\begin{lemma}\label{unique01}
			Given $G_{T,\gamma}$, suppose there is an edge $e$ with $\ell(e)$ unique. Then if there is a set of matchings $S$ and a lattice point $\eta=\sum_{M\in S}c_Mw^M$ with $\sum_{M\in S}c_M=1$ and $c_M>0$ for all $M\in S$, the coordinate $(w^M)_{\ell(e)}$ must be the same for all $M\in S$. 
		\end{lemma}
		\begin{proof}
			Each $w^M_{\ell(e)}$ is either $0$ or $1$, depending on whether or not $e$ was included in the perfect matching. Restricting the convex combination to the coordinate $\ell(e)$ yields a convex combination of $0$s and $1$s that is equal to an integer. Since each $c_M>0$, this is only possible if it is a convex combination of only $0$s or only $1$s. 
		\end{proof}
		Now, we give three lemmas on emptiness and non-emptiness, which cover all Newton polytopes $N(T, \gamma)$.
		\begin{lemma}\label{DBH-arc}
			Let $T$ be an ideal triangulation of $\dpoly$. If $\gamma$ is an arc that intersects each arc of $T$ at most once, then $P(G_{T,\gamma})$ is empty. 
		\end{lemma}
		\begin{proof}

			We will prove this by induction on the number of square in $G_{T,\gamma}$. For the base case where there is a single square, there is some edge $e$ with $\ell(e)$ unique. Given a set of matchings $S$ as in Lemma~\ref{unique01}, that lemma implies that either $e\in M$ or $e\notin M$ for all $M\in S$. So $|S|=1$, implying emptiness. 
			
			Next, suppose $P(G_{T,\gamma})$ is empty for all $G_{T,\gamma}$ with at most $t$ boxes. Suppose there is an arc $\tau$ such that $G_{T,\tau}$ has $t+1$ boxes, and suppose there exists a set of matchings $S$ as in Lemma \ref{unique01}. By Lemma \ref{unique}, there is a square $\alpha$ at an end of $G_{T,\tau}$ and a boundary edge $e\in\alpha$ with $\ell(\alpha)$ unique in $G_{T,\tau}$. By Lemma \ref{unique01}, $M|_{\alpha}$ is the same for all $M\in S$. Let $\tau'$ be an arc such that $G_{T,\tau'}$ is equal to $G_{T,\tau}$ with $\alpha$ removed (the pathological boxed case of \ref{fig:shifts} is excluded by the hypothesis of this lemma). Letting $a,b,d$ be the three boundary edges of $\alpha$ as in Figure \ref{diagrams}, for each $M\in S$ define the matching $M'$ of $G_{T,\tau'}$ as follows: if $a\in M$ for all $M\in S$, then $M'=M\setminus\{a\}$; otherwise $M'=M\cup\{c\}\setminus\{b,d\}$. Since $\sum_{M\in S} c_Mw^{M'}$ is a lattice point, the inductive hypothesis implies $|S|=1$, so $P(G_{T,\tau})$ is empty.
		\end{proof}
		\begin{lemma}\label{DBH-notched}
			Let $T$ be an ideal triangulation of $\dpoly$. If $\gamma$ is a notched arc, then $P(G_{T,\gamma})$ is empty.
		\end{lemma}
		\begin{proof}

			Consider $T$ as an ideal triangulation. If $k=1$, by Remark~\ref{rmk:easyNotchedExpansions} it suffices to show that $P(G_{T,\gamma^{0}})$ is empty, which follows from Lemma~\ref{DBH-arc}.
			
			If $k>1$, let $\lambda$ be the loop enclosing $\gamma$ and let $\rho$ be the plain radius. If $\lambda$ intersects each arc of $T$ at most once, the $\gamma$-symmetry condition is vacuous and $P(G_{T,\gamma})$ is empty by Lemma \ref{DBH-arc}. Thus, we assume that $\lambda$ intersects at least one arc more than once, without loss of generality including $\varepsilon_1$. 
			
			Let $\alpha$ and $\alpha'$ be the isomorphic squares on the ends of $G_{T,\rho,1}$ and $G_{T,\rho,2}$ respectively. By Lemma \ref{unique} we have boundary edges with $e\in\alpha$ and $e'\in\alpha'$ such that $\ell(e)=\ell(e')$ does not label any other edges in $G_{T,\gamma}$. Either $e,e'\in M$ or $e,e'\notin M$ for any $\gamma$-symmetric matching $M$. Letting $G'$ be the snake graph obtained by removing $\alpha,\alpha'$ from $G_{T,\gamma}$, we can use the method of Lemma \ref{DBH-arc} simultaneously on both ends to reduce any set of matchings of $G_{T,\gamma}$ satisfying the hypothesis of Lemma \ref{unique01} to a set of matchings on $G'$ satisfying the same hypothesis.  
			
			If $\lambda$ intersects at least $2$ non-radius arcs of $T$, then $G'=G_{T,\lambda'}$ where $\lambda'$ is a loop with both endpoints moved as shown in Figure \ref{fig:shifts}, and $\lambda'$ intersects one fewer arc of $T$. However, if $\lambda$ intersects only $\varepsilon_1$ and radii of $T$, then we find ourselves in the pathological boxed case; the $\gamma$-symmetry condition is vacuous for $G'$, so we directly show $P(G')$ is empty for such $G'$. Let $\sum_{M\in S} c_M\overline{w}^{M}$ satisfy the hypothesis of Lemma \ref{unique01}. Then each of $\varepsilon_2,..,\varepsilon_k$ is contained in all of or none of $M\in S$, since they each label a unique edge. If at least one of these is included, then each $M\in S$ is uniquely determined; if none of the interior edges of $G'$ are included, then there are two possible matchings $M\in S$ but they have the same matching vector. Thus, $|S|=1$, proving that $P(G')$ is empty.
		\end{proof}
		\begin{lemma}\label{DB-nEmpty}
			Let $T$ be an ideal triangulation of $\dpoly$. If $\gamma\notin T$ is a plain arc intersecting some arc of $T$ more than once, then $P(G_{T,\gamma})$ is not empty.
		\end{lemma}
		\begin{proof}
			%
			
			Without loss of generality, $\gamma$ intersects $\varepsilon_1$ twice. If $k>1$, then the squares corresponding to $\rho_1$ and $\rho_k$ are non-corner squares, so there is a matching that includes a pair of opposite edges of each of the two squares labeled $\varepsilon_1$. There are four choices of which pair is chosen from each square, leaving the rest of the matching the same. These four matchings yield three matching vectors, one of which is the midpoint of the other two, so $P(G_{T,\gamma})$ is not empty.
			
			If $k=1$, then there are four matchings of the three tiles labeled $\rho_1$, $\varepsilon_1$, $\rho_1$ (see Figure \ref{fig:triple-tile}) again with three distinct matching vectors with one the midpoint of the other two, and these four matchings can be extended identically to the rest of the diagram to show that $P(G_{T,\gamma})$ is not empty.
		\end{proof}
		Combining these lemmas yields the desired classification.
		\begin{theorem}\label{DB-Emptiness}
			Let $\mca^{bd}$ be a type $D$ cluster algebra with boundary frozens at seed $\Sigma$, corresponding to a tagged triangulation $T$. Let $\gamma$ be a tagged arc and let $N^{bd}(T, \gamma)$ be the Newton polytope of $x_\gamma$, written as a Laurent poynomial in $\Sigma$.
			\begin{itemize}	
				\item If $\gamma$ is plain and $T \neq T^p$, then $N^{\text{bd}}(T,\gamma)$ is empty if and only if $\gamma$ intersects no arc of $T$ more than once. 
				\item If $\gamma$ is plain and $T=T^p$, then $N^{\text{bd}}(T,\gamma)$ is empty if and only if $\gamma$ intersects no arc of $T^\circ$ more than once.
				\item If $\gamma$ is notched, then $N^{\text{bd}}(T,\gamma)$ is empty. 
			\end{itemize} 
		\end{theorem}
		\begin{proof}
			If $T$ has all radii plain, then $T\neq T^{p}$ and $T$ corresponds to an ideal triangulation $T'$. As before, we use Lemma \ref{lem:saturationEquiv} to note that emptiness of $N^{\text{bd}}(T,\gamma)$ is equivalent to emptiness of $P^{\text{bd}}(G_{T',\gamma})$. By Lemma \ref{DBH-arc}, Lemma \ref{DBH-notched}, and Lemma \ref{DB-nEmpty}, the stated condition is a precise characterization of emptiness of $P(G_{T',\gamma})$. 
			
			If $T$ has all radii notched, then $T\neq T^{p}$ and $T^p$ corresponds to an ideal triangulation $T'$. If $\gamma$ is not a radius, then by Remark~\ref{rmk:easyNotchedExpansions} emptiness of $N^{\text{bd}}(T,\gamma)$ is equivalent to emptiness of $P^{\text{bd}}(G_{T',\gamma})$, and we can proceed as above. If $\gamma$ is a radius, then emptiness of $N^{\text{bd}}(T,\gamma)$ is equivalent to emptiness of $P^{\text{bd}}(G_{T',\gamma^p})$, and the latter is always empty. 
			
			Finally, if $T$ has a notched and un-notched radius $\rho,\rho^{\bowtie}$, then $T=T^p$ and $T$ corresponds to an ideal triangulation $T'=T^{\circ}$ replacing $\rho^{\bowtie}$ with a loop $\lambda$. Emptiness of $N^{\text{bd}}(T,\gamma)$ is equivalent to emptiness of $P^{\text{bd}}(T^{\circ},\gamma)$ by Lemma~\ref{lem:saturationEquiv}, and the latter is only non-empty if $\gamma$ intersects some arc of $T^{\circ}$ more than once. 
		\end{proof}
		\subsection{Principal coefficients}
		
		As with Type $A$, the results for Type $D$ with principal coefficients are very similar to those with boundary frozen variables. In this section, let $\pi=\pi^{\text{pc}}$.
		
		Having developed the machinery of the principal perfect matching polytope, the proofs are almost identical to the boundary coefficients case. But first we need one lemma about the structure of the principal coefficient variables in Type $D$:
		\begin{lemma}\label{bottom-symmetric}
			For $T$ an ideal triangulation of $\dpoly$ and $\gamma$ a loop, the bottom matching of $G_{T,\gamma}$ is $\gamma$-symmetric. 
		\end{lemma}
		\begin{proof}
			%
			
			If $\gamma$ only intersects radii, then the $\gamma$-symmetry condition holds vacuously. If $\gamma$ intersects a single non-radius arc of $T$, then $G_{T,\gamma}$ looks like Figure \ref{fig:staircase}. The bottom matching includes two boundary edges from one end square and one boundary edge from the other, which is $\gamma$-symmetric. If $\gamma$ intersects more non-radius arcs of $T$, then $G_{T,\gamma}$ will look like Figure \ref{fig:staircase} extended on both ends with isomorphic squares, and we see inductively that the bottom matching is always $\gamma$-symmetric. 
		\end{proof}
		We will now state the results and note how the proofs differ from the boundary coefficients case.
		\begin{theorem}\label{DP-saturated}
			Let $\mca^{\text{pc}}$ be a type $D$ cluster algebra with principal coefficients at a seed $\Sigma$. Then the Newton polytope of any cluster variable, written as a Laurent polynomial in $\Sigma$, is saturated.
		\end{theorem}
		\begin{proof}
			As usual, we may assume that $\Sigma=\Sigma_{T'}$ where $T'$ is a tagged triangulation corresponding to an ideal triangulation $T$. 
			
			To prove Lemma \ref{chip1} for principal coefficients, note that if the first square $\alpha$ of $G_{T,\gamma}$ has an edge with a unique label, then no other square of $G_{T,\gamma}$ is labeled $\ell(\alpha)$. Therefore, all $\overline{v}=(\overline{v}_1, \overline{v}_2)\in\overline{P}^{\text{pc}}(G_{T,\gamma})$ satisfy $(\overline{v}_2)_{\ell(\alpha)}=(\pi(\overline{v})_2)_{\ell(\alpha)}$. Then for any lattice point $\eta\in P^{\text{pc}}(G_{T,\gamma})$, each $\overline{\eta}\in\pi^{-1}(\eta)$ has $(\overline{\eta}_2)_{\ell(\alpha)}\in\mathbb{Z}$. By Corollary \ref{PMPC}, each $\overline{v}\in \overline{P}^{\text{pc}}(G_{T,\gamma})$ has $(\overline{v}_2)_{\ell(\alpha)}=(\overline{v}_1)_{\ell(e)}$ for some edge $e$ of $\alpha$. Therefore, $(\overline{v}_1)_{\ell(e)}$ is an integer, and we can proceed exactly as in Lemma \ref{chip1} to extend a matching of a smaller saturated snake graph to a matching with weight vector $\pi(\overline{\eta})$. 
			
			Likewise we can prove the analogue of Lemma \ref{chip2}. If two isomorphic squares $\alpha,\alpha'$ at the ends of $G_{T,\gamma}$ have corresponding edges $e\in\alpha,e'\in\alpha'$ with $\ell(e)=\ell(e')$ labeling no other edges, then there are no other squares labeled by $\ell(\alpha)=\ell(\alpha')$. Exactly as above, we can convert equations in the coordinates of $\overline{v}_2$ into equations involving coordinates of $\overline{v}_1$, and then apply the proof of Lemma \ref{chip2} to construct a matching with weight vector $\pi(\overline{v})$.
			
			For Lemma \ref{pseudo-loop}, since the squares of $G'$ are distinct, the coordinates of $\overline{v}_2$ must be $0$ or $1$, so $(\overline{v}_1)_{e}$ must be an integer for all boundary edges $e$ of $G'$. Then the same must be true for interior edges $e$ of $G'$, showing that $\pi(\overline{v})$ is the weight vector of a matching. 
			
			For Lemma \ref{loop1}, for $k>1$ we apply precisely the same reasoning to show that if $\alpha$ is a square of $G_{T,\gamma}$ with $\ell(\rho)$ a radius and $e$ is an edge of $\alpha$ on the boundary of $G_{T,\gamma}$, then $(\overline{v}_1)_{\ell(e)}$ is $0$ or $1$ for all $\overline{v}$. Noting that the squares corresponding to $\rho_1$ and $\rho_k$ are non-corner squares of $G_{T,\gamma}$, Corollary \ref{PMPC} implies that $(\overline{v}_1)_{\varepsilon_2},...,(\overline{v}_1)_{\varepsilon_k}$ are each $0$ or $1$. Then casework shows that both $(\overline{v}_2)_{\varepsilon_1}$ and $(\overline{v}_2)_{\varepsilon_1'}$ must be integers too, so $\pi(\overline{v})$ corresponds to a matching. If $k=1$, then $G_{T,\gamma}$ consists of three tiles, and we can check directly that no nontrivial convex combination of matching vectors yields a lattice point.
			
			Lemma \ref{unique} and Lemma \ref{strong-induction} can be applied directly to the principal coefficients case (using our work above for the latter), so we can put all of this together to replicate the proof of Lemma \ref{DB-saturated}. 
			
			For Lemma \ref{DB-notched}, note that by Lemma \ref{bottom-symmetric} the coordinates $(v_2)_{\ell(\alpha)}$ corresponding to squares $\alpha$ in $G_{T,\gamma,1}$ and $G_{T,\gamma,2}$ determine each other by the $\gamma$-symmetry condition. Therefore, we can adapt the above arguments to $G_{T,\gamma,1}$ and $G_{T,\gamma,2}$ to induct simultaneously on both ends of $G_{T,\gamma}$, proceeding as in the previous case.
		\end{proof}
		
		Next, we characterize when Newton polytopes for type $D$ cluster variables with principal coefficients are empty. As with the boundary frozens case, not all Newton polytopes are empty. Recall that if $T$ is a tagged triangulation of $\dpoly$, $T^p$ denotes the triangulation obtained by switching the tagging of each radius. Further, $T=T^p$ if and only if $T$ corresponds to an ideal triangulation with a loop.
		
		\begin{theorem}\label{DP-emptiness}
			Let $\mca^{pc}$ be a type $D$ cluster algebra with principal coefficients at a seed $\Sigma$, corresponding to a tagged triangulation $T$. Let $\gamma$ be a tagged arc and let $N^{pc}(T, \gamma)$ be the Newton polytope of $x_\gamma$, written as a Laurent poynomial in $\Sigma$.
			\begin{itemize}
				\item If $\gamma$ is plain and $T \neq T^p$, then $N^{\text{pc}}(T,\gamma)$ is empty if and only if $\gamma$ intersects no arc of $T$ more than once. 
				\item If $\gamma$ is plain and $T=T^p$, then $N^{\text{pc}}(T,\gamma)$ is empty if and only if $\gamma$ intersects at most one arc of $T^\circ$ more than once.
				\item If $\gamma$ is notched, then $N^{\text{pc}}(T,\gamma)$ is empty. 
			\end{itemize} 
		\end{theorem}
		\begin{proof}

			First, consider the case that $\gamma$ is plain. Note that Lemma \ref{unique01} also applies to squares $\alpha$ of $G_{T,\gamma}$ with $\ell(s)$ unique. 
			
			As in our proof of Theorem \ref{DP-saturated}, if the first square $\alpha$ of $G_{T,\gamma}$ has an edge $e$ on the boundary of $G_{T,\gamma}$ with $\ell(e)$ unique, then $\ell(\alpha)$ is also unique. Then given a set of matchings $S$ satisfying the hypothesis of Lemma \ref{unique01}, each of the three edges of $\alpha$ on the boundary of $G_{T,\gamma}$ are either included or excluded in all $e\in M$. So as in the proof of Lemma \ref{DBH-arc} we can reduce each $M\in S$ to a matching $M'$ on a smaller snake graph, with $\sum_{M\in S}c_M\overline{w}_{\text{pc}}^{M'}$ equal to a lattice point.
			
			If $\gamma$ intersects each arc of $T$ at most once, then we can use this method to reduce the snake graph until we have removed all squares, implying emptiness by induction. If $T$ contains a loop that $\gamma$ intersects twice, then we can reduce the snake graph to the three tiles corresponding to $\varepsilon_1$, $\rho_1$, $\varepsilon_1$. We saw in the proof of Theorem \ref{DP-saturated} that for matchings $M$ of this graph, $\{\overline{w}_{\text{pc}}^M\}_{M}$ have no convex combination equal to a lattice point, proving emptiness. 
			
			Conversely, the construction of Lemma \ref{DB-nEmpty} also violates emptiness for principal coefficients if $T$ has no loop and $\gamma$ intersects some arc more than once. If $T$ has a loop and $\gamma$ intersects at least two arcs more than once, let $\tau$ be the arc intersected before/after the loop $\varepsilon_1$. Then we can adapt the construction of Lemma \ref{DB-nEmpty} with $\tau$ instead of $\varepsilon_1$, i.e., constructing the four matchings using pairs of opposite sides from each of the two squares corresponding to $\tau$. These can all be extended identically to the rest of the snake graph, and one of the resulting principal matching vectors is the midpoint of two others. 
			
			Now, let $\gamma$ be notched. By Lemma \ref{bottom-symmetric}, we can use the $\gamma$-symmetry condition to apply our work above to $G_{T,\gamma,1}$ and $G_{T,\gamma,2}$. Letting $G'$ be $G_{T,\gamma}$ with isomorphic subgraphs removed from each end, we can reduce $\sum_{M\in S}c_M\overline{w}_{\text{pc}}^{M}$ with $M\in S$ matchings of $G_{T,\gamma}$ to $\sum_{M\in S}c_M\overline{w}_{\text{pc}}^{M'}$ with each $M'$ a matching of $G'$. To finish proving this analogue of Lemma \ref{DBH-notched}, letting $G'$ consist only of squares labeled by radii, note that each square has a unique label. So if $S$ satisfies the hypothesis of Lemma \ref{unique01}, then all matchings in $S$ agree on the boundary edges of $G'$. And as we saw in the proof of Theorem \ref{DP-saturated}, the boundary edges in a matching of $G'$ determine the matching, proving emptiness. 
			
		\end{proof}
		
		\begin{remark}
			It is possible to construct a pair of inverse affine maps $\overline{P}^{\text{pc}}(G_{T,\gamma})\longleftrightarrow \overline{P}(G_{T,\gamma})$, proving that the polytopes are combinatorially equivalent. Furthermore, both maps take lattice points to lattice points. However, they do not descend to well-defined maps $P^{\text{pc}}(G_{T,\gamma})\longleftrightarrow P(G_{T,\gamma})$, since affine maps do not commute with adding coordinates. Nevertheless, perhaps investigating these maps could be a fruitful avenue for future work.
		\end{remark}
 	\section{Cluster variable Newton polytopes for other surfaces}
		Now, we examine cluster algebras from surfaces other than $\apoly$ and $\dpoly$. We will not define them here, but similar snake graph expansion formulas hold for cluster variables, see \cite{MSW}.
		
		Types $A$ and $D$ are the only cluster algebras from surfaces that are finite type, and are in general unusually well-behaved from a combinatorial perspective. This holds for Newton polytopes of cluster variables as well. Many of the results in the previous sections fail for arbitrary cluster algebras from surfaces.
		
		\subsection{Counterexamples}
		We give some examples demonstrating that saturation does not hold in general for cluster algebras from surfaces with boundary frozen variables. 
		
		\begin{example}[Annulus]
			\begin{figure}
				\begin{minipage}{0.45\linewidth}
					\centering
					\includegraphics[width=0.7\linewidth]{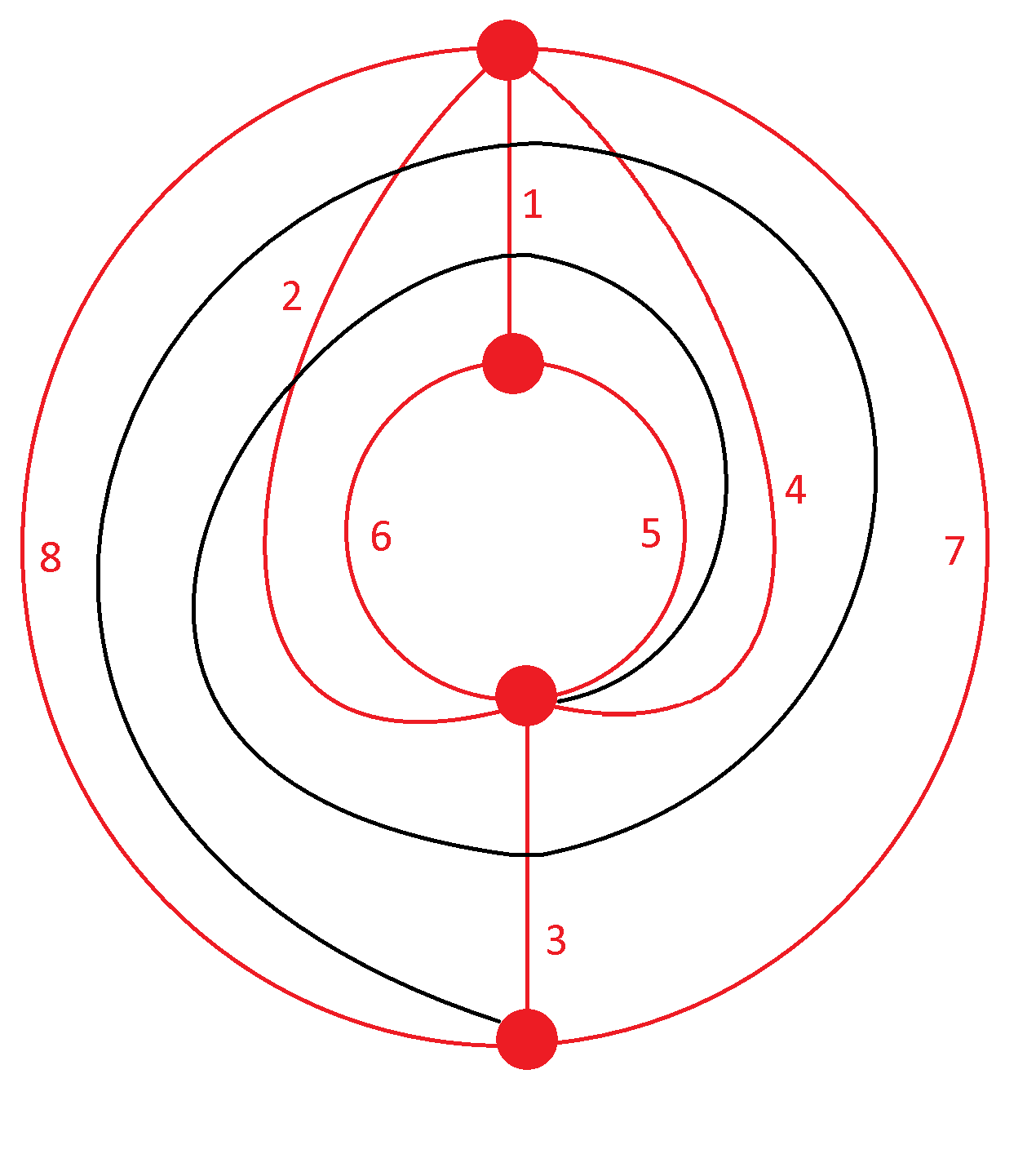}
				\end{minipage}
				\begin{minipage}{0.45\linewidth}
					\centering
					\includegraphics[width=\linewidth]{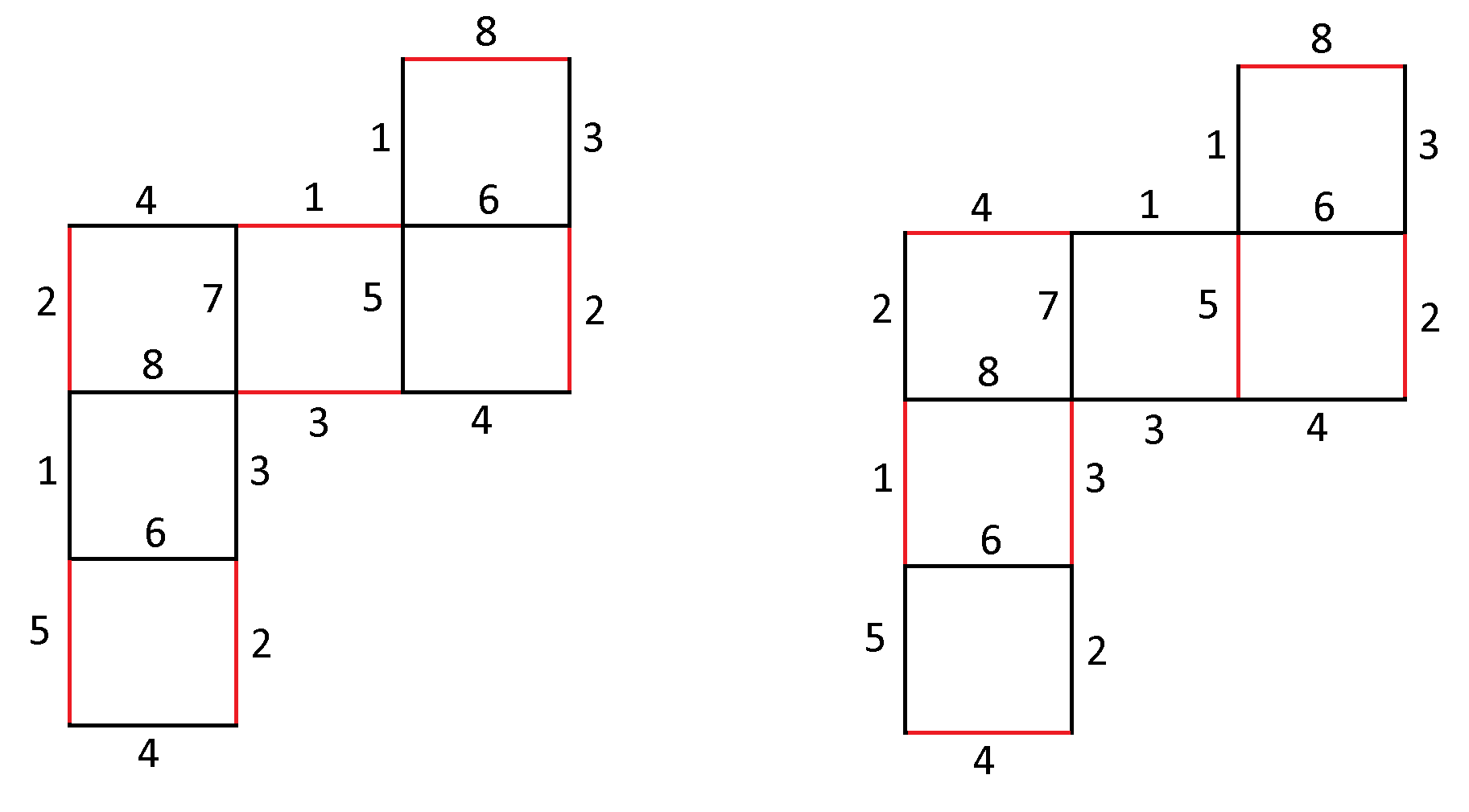}
				\end{minipage}
				
				\caption{On the left is an arc $\tau$ (black) and a triangulation $T$ of an annulus. On the right are two matchings of the snake graph $G_{T,\gamma}$. The midpoint of the weight vectors of these matchings does not correspond to a matching.}
				\label{fig:annulus-triangulation}
			\end{figure}
			For an annulus with two marked points on each boundary component, let triangulation $T$ and arc $\gamma$ be as shown in Figure \ref{fig:annulus-triangulation} on the left. The matchings of $G_{T,\gamma}$ on the right yield weight vectors $(1,1,1,2,1,0,0,1)$ and $(1,3,1,0,1,0,0,1)$. Their midpoint, $(1,2,1,1,1,0,0,1)$, does not correspond to a matching. Thus, $P(G_{T,\gamma})$ is not saturated.
		\end{example}
		\begin{example}[Punctured Torus]
			For a torus with one marked point in the interior, let triangulation $T$ and arc $\gamma$ be as shown in Figure \ref{fig:punctured-torus}. For $G_{T,\gamma}$, shown on the right, one can check that each matching will include an even number of edges with each label. But there are pairs of weight vectors whose midpoints have odd coordinates (e.g., let one matching include two opposite sides of one square, and let the other matching be the same but with the other two sides of that square), so these midpoints cannot correspond to matchings. Thus, $P(G_{T,\gamma})$ is not saturated.
			
			\begin{figure}
				\begin{minipage}{0.4\linewidth}
					\centering
					\includegraphics[width=0.7\linewidth]{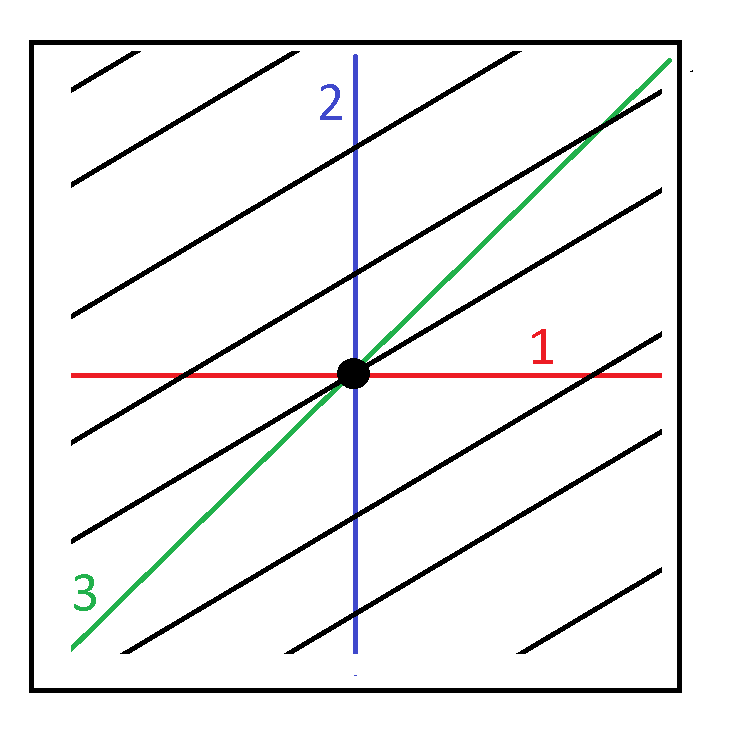}
				\end{minipage}
				\begin{minipage}{0.4\linewidth}
					\centering 
					\includegraphics[width=0.7\linewidth]{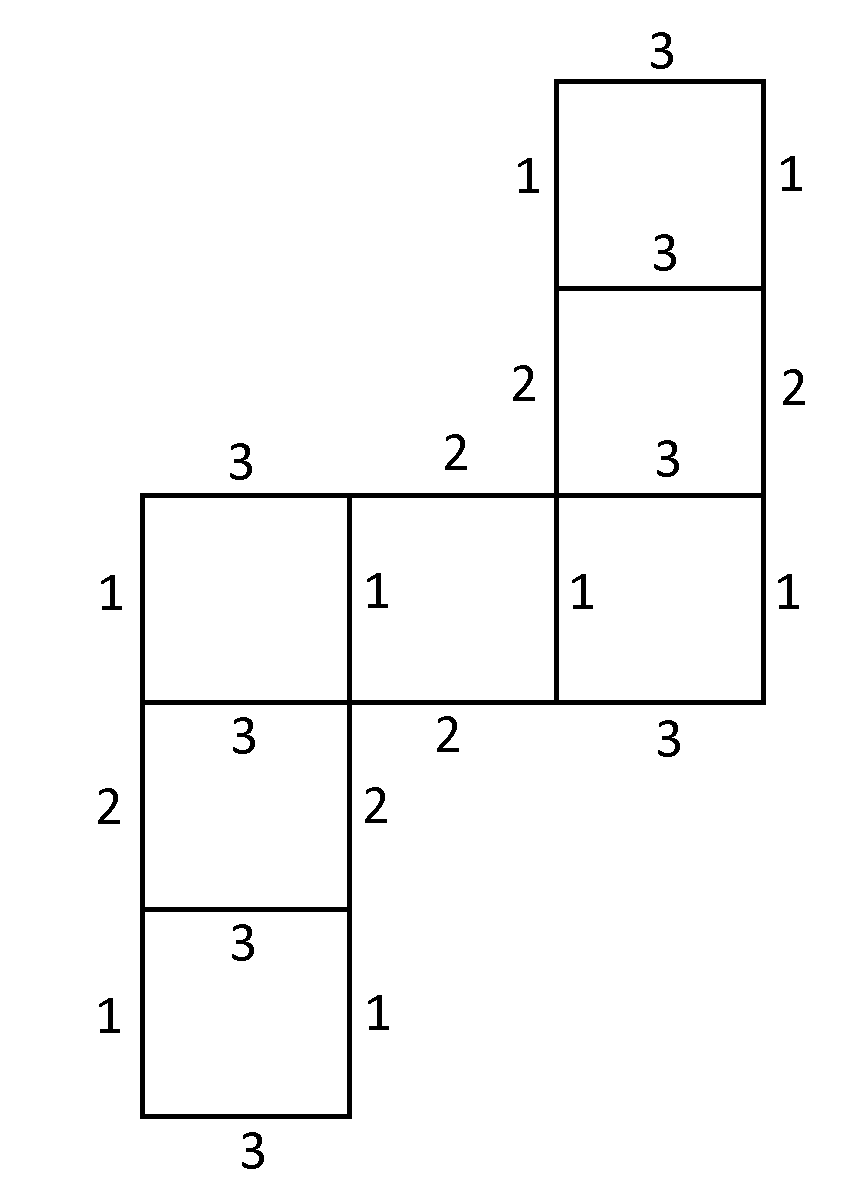}
				\end{minipage}
				\caption{On the left is an arc $\gamma$ (black) and a triangulation $T$ of a punctured torus. On the right is the snake graph $G_{T,\gamma}$.}
				\label{fig:punctured-torus}
			\end{figure}
		\end{example}
		\begin{example}[Twice Punctured Torus]
			For a torus with two marked points in the interior, let triangulation $T$ and arc $\gamma$ be as shown in Figure \ref{fig:twice-punctured-torus-triangulation} on the left. The matchings of $G_{T,\gamma}$ on the right yield weight vectors $(2,0,1,1,1,2)$ and $(0,2,1,1,1,2)$. Their midpoint, $(1,1,1,1,1,2)$, does not correspond to a matching. Thus, $P(G_{T,\gamma})$ is not saturated.
			\begin{figure}
				\begin{minipage}{0.4\linewidth}
					\centering
					\includegraphics[width=0.7\linewidth]{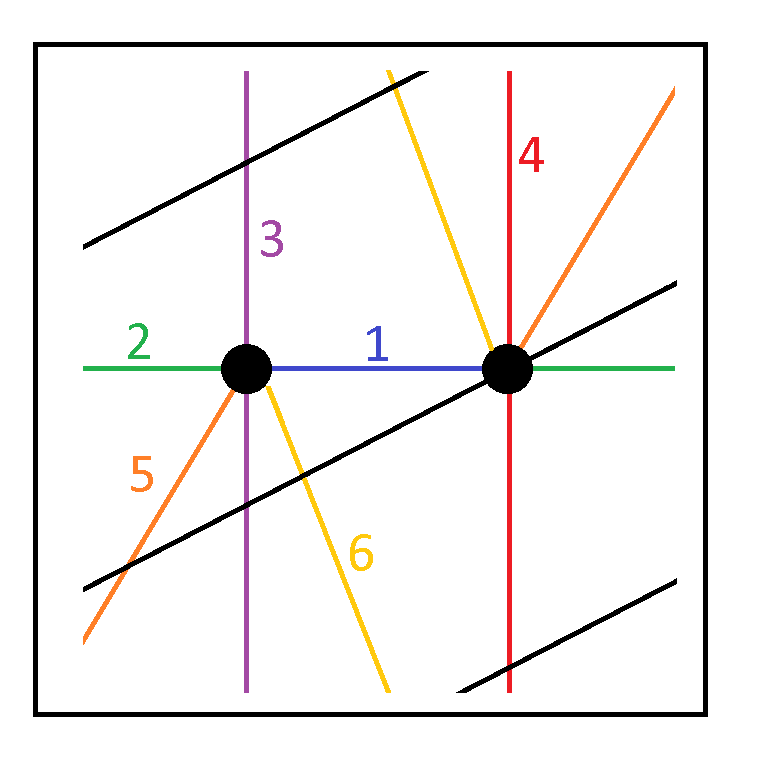}
				\end{minipage}
				\begin{minipage}{0.55\linewidth}
					\centering
					\includegraphics[width=\linewidth]{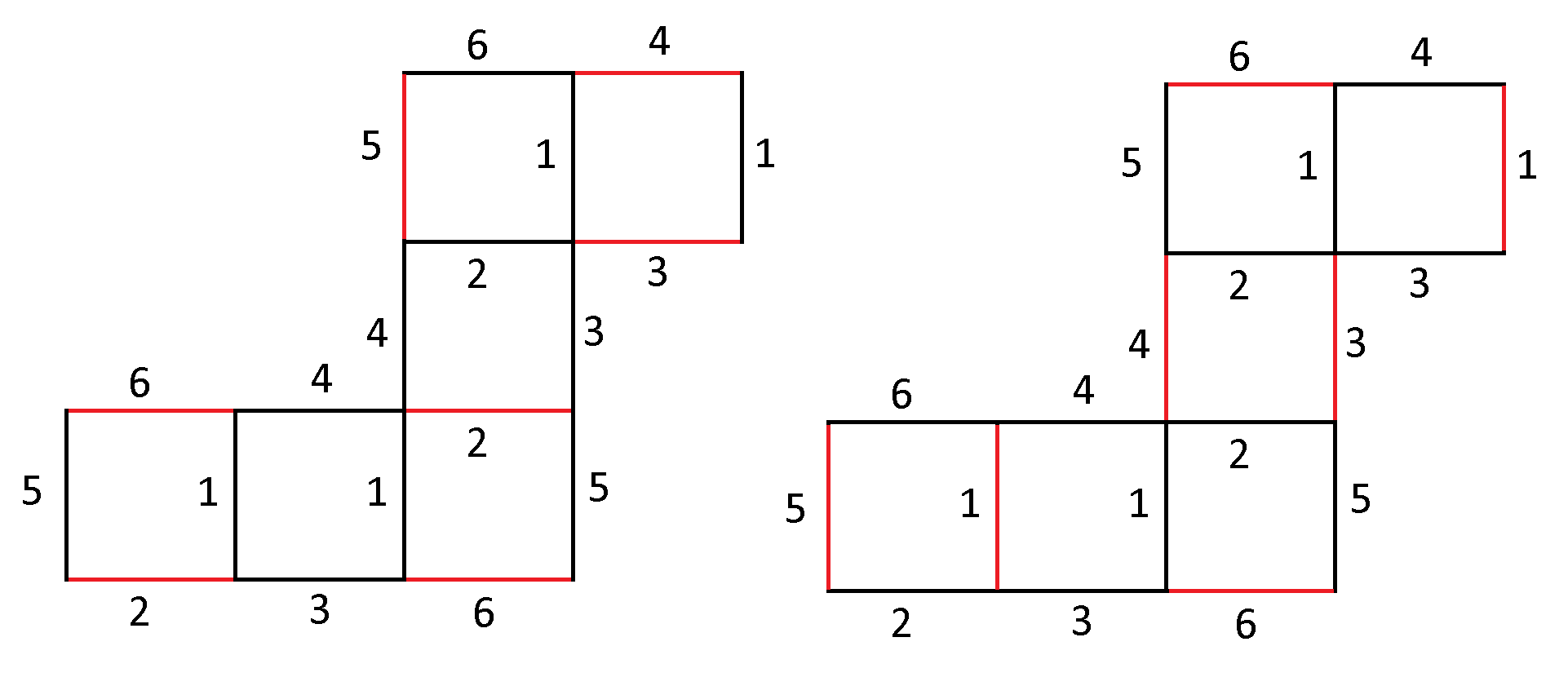}
				\end{minipage}
				\caption{On the left is an arc $\gamma$ (black) and a triangulation $T$ of a twice-punctured torus. On the right are two matchings of the snake graph $G_{T,\gamma}$. The midpoint of the weight vectors of these matchings does not correspond to a matching.}
				\label{fig:twice-punctured-torus-triangulation}	
			\end{figure}
		\end{example}
		These examples were computed using a combination of original Python code (available upon request) and the online Polymake software.
		\subsection{Conjectures}
		Given these counterexamples and the exceptionalness of finite type cluster algebras, we conjecture the following.
		\begin{conj}\label{special}
			Let $\mca^{\text{bd}}$ be a cluster algebra from a surface $(S,M)$ with boundary frozen variables. Then $N^{\text{bd}}(T,\gamma)$ is saturated for all $T,\gamma$ only if $(S,M)$ is $\apoly$ or $\dpoly$.
		\end{conj}
		Noting that our proofs in types $A$ and $D$ made use of certain labels appearing at most once or twice in $G_{T,\gamma}$, perhaps saturation fails when snake graphs get too big and labels repeat arbitrarily many times --- which can always occur in non-finite type cluster algebras from surfaces. Thus, we have a stronger version of Conjecture \ref{special}.
		\begin{conj}\label{stronger}
			Let $\mca^{\text{bd}}$ be a cluster algebra from a surface $(S,M)$ with boundary frozen variables. There exists $B\in\mathbb{N}$ such that if $G_{T,\gamma}$ has more than $B$ squares, then $N^{\text{bd}}(T,\gamma)$ is not saturated.
		\end{conj}
		Note that Conjecture \ref{stronger} is vacuously true if $(S,M)$ is $\apoly$ or $\dpoly$ since the number of squares in $G_{T,\gamma}$ is bounded in these finite type cases.
		
		We note that we have not found a counterexample to Fei's conjecture: that cluster variable Newton polytopes in cluster algebras with principal coefficients are saturated.
	\nocite{*}
	\bibliography{bibliography}{}
	\bibliographystyle{plain}
\end{document}